\documentclass{article}
\usepackage[utf8]{inputenc}
\usepackage{fullpage}
\usepackage{hyperref}
\usepackage{xcolor}
\usepackage{amsmath,amsfonts,amssymb}
\usepackage{graphicx}%
\usepackage{subcaption}
\usepackage{mathrsfs}
\usepackage[margin=1.0in]{geometry}
\usepackage{appendix}
\usepackage{accents}

\usepackage{accents}

\usepackage{amsthm}
\usepackage[shortlabels]{enumitem}
\newtheorem{theorem}{Theorem}[section]
\newtheorem{remark}[theorem]{Remark}

\newtheorem{lemma}[theorem]{Lemma}
\newtheorem{corollary}[theorem]{Corollary}

\newcommand{\sfR}{{\mathsf R}}

\newcommand{\mQ}{\mathcal{Q}}
\newcommand{\eps}{\varepsilon}
\newcommand{\E}{\mathbb E}
\newcommand{\Rm}{\mathbb R}
\newcommand{\dint}{\displaystyle\int}

\newcommand{\tps}{\mathrm{t}}
\newcommand{\aver}[1]{\langle {#1} \rangle}

\newcommand{\sL}{{\mathsf L}}

\title{Long distance propagation of wave beams in paraxial regime}
\author{Guillaume Bal \thanks{Departments of Statistics and Mathematics and Committee on Computational and Applied Mathematics, University of Chicago, Chicago, IL 60637; guillaumebal@uchicago.edu} \and Anjali Nair \thanks{Department of Statistics and Committee on Computational and Applied Mathematics, University of Chicago, Chicago, IL 60637; anjalinair@uchicago.edu}}
\date{\today}

\begin{document}

\maketitle

\begin{abstract}
This paper concerns the propagation of high frequency wave-beams in highly turbulent atmospheres. Using a paraxial model of wave propagation, we show in the long-distance weak-coupling regime that the wavefields are approximately described by a complex Gaussian field whose scintillation index is unity. This provides a model of the speckle formation observed in many practical settings. The main step of the derivation consists in showing that closed-form moment equations in the It\^o-Schr\"odinger regime are still approximately satisfied in the paraxial regime. {\color{black}This is done via Duhamel expansions, under the assumption that the random medium has Gaussian statistics.} The rest of the proof is then an extension of results derived in \cite{bal2024complex}.
\end{abstract}


\section{Introduction}

\label{sec:intro}

A widely accepted conjecture in the physical literature states that high frequency waves eventually become complex circular Gaussian distributed after propagating over long distances through randomly varying media~\cite{goodman1976some, sheng1990scattering}. The real and imaginary parts of the wave field then become independent Gaussian distributed. The wave intensity follows an exponential distribution asymptotically, implying a scintillation index converging to unity. This forms a model for speckle patterns consistent with experimental observations~\cite{andrews2001laser, carminati2021principles}. This conjecture is not supported mathematically for general models of wave propagation based on Helmholtz equations. The highly complex nature of wave propagation in random media  makes mathematical analyses of asymptotic models inherently difficult; see \cite{BKR-KRM-10, garnier2018multiscale} for a review of existing results.

A simpler setting is that of the paraxial approximation, where the wave-field retains its beam-like structure for long distances while backscattering is neglected; see \cite{bailly1996parabolic,garnier2009coupled} for theoretical justifications of such a regime. A further approximation is the It\^{o}-Schr\"{o}dinger white noise limit of the paraxial model, which using It\^o calculus, is characterized by closed form PDEs for statistical moments of all orders; see  \cite{dawson1984random} for an analysis of the white noise model and \cite{fannjiang2004scaling} for a derivation starting from the paraxial approximation. 

While analytical solutions to these PDEs are not known in general, they form an important tool to carry out long-distance asymptotic analyses. The structure of these PDEs up to the fourth moment was used in a series of articles~\cite{garnier2014scintillation, garnier2016fourth, garnier2022scintillation, garnier2023wave, garnier2023speckle} to understand scintillation of the energy of the wavefield. 

The Gaussian conjecture was proved for the It\^{o}-Schr\"{o}dinger model of wave propagation in the scintillation (weak-coupling) regime in \cite{bal2024complex}. The derivation is based on a careful analysis of these PDEs for statistical moments of arbitrary order. That work identifies two scaling limits. In the first one called the {\em kinetic regime}, random fluctuations with small amplitude accumulate over medium-long distances of propagation so that the effective turbulence strength is of order $O(1)$. In the longer-distance, second scaling called the {\em diffusive regime}, the cumulative strength of turbulence is larger than $O(1)$. In the latter scaling, the mean wave-field vanishes while its higher statistical moments are indeed consistent with those of a circular Gaussian random variable. The observed wavefield intensity then follows an exponential distribution, characterising the speckle patterns observed in experiments.

From a technical standpoint, the analysis of such moments in \cite{bal2024complex} is primarily carried out in the Fourier domain, with the error controlled in the total variation sense. The main advantage of this topology is that it allows us to obtain error estimates in the physical domain in the uniform sense. Other related scaling limits have also been obtained for a Fourier-domain representation of a phase-compensated wavefield in~\cite{bal2011asymptotics, gu2021gaussian}. 

We note that the Gaussian conjecture is not always true, even in the It\^o-Schr\"odinger regime. It has been shown that in the so-called spot dancing regime, the wave field approximately follows a Rice-Nakagami distribution instead for Gaussian incident beams~\cite{furutsu1973spot, dawson1984random}. Such a regime is characterized by relatively narrow beams travelling through very strong turbulence over relatively short distances. The wave field then experiences significant beam wander with the overall profile of the beam diffracting as in the uniform medium. In contrast, the scintillation regime is characterized by broad incident beams propagating through weak turbulence over long distances. In such a regime, the wavefield decorrelates rapidly along the lateral cross-section of the beam, which is consistent with  speckle patterns observed in physical experiments.     

\medskip

The main objective of this paper is to prove the Gaussian conjecture in the scintillation regime starting from the paraxial approximation of the scalar Helmholtz equation. As in the derivation for the It\^{o}-Schr\"{o}dinger equation in \cite{bal2024complex}, this requires sufficiently rapid decorrelations of the random fluctuations along the direction of propagation in the rescaled variables. The presence of long range correlations leads to considerably different regimes~\cite{gomez2017fractional, borcea2023paraxial}. These requirements are encoded in the following scaling of wave propagation.

The starting microscopic model of wave-beam propagation in turbulent media is the following scalar Helmholtz equation:
\begin{equation}\label{eqn:Helmholtz}
      \Big( \partial^2_z + \Delta_x + k_0^2n^2(z,x)\Big) p(z,x) = {\color{black}+}u_0(x) \delta_0{\color{black}'}(z)\,, 
\end{equation}
where $k_0$ is the carrier wave number of the source, $z\in\Rm$ is depth and $x\in\Rm^d$ for $d\geq1$ denotes transverse spatial variables.  Suppose $n^2(z,x)=1+\nu(z,x)$ where $\nu$ is a random coefficient modeled as a mean zero Gaussian random process. We consider a high frequency long propagation scaling given by
\begin{equation}\label{eq:scalings}
  z \to \frac{\eta z}{\eps\theta},\qquad x \to x,\qquad k_0\to \frac{k_0}\theta,\qquad 
  \nu \to \frac{\eps^{\frac12}\theta^{\frac32}}{\eta^{\frac32}} \nu
  \qquad u_0\to u_0^\theta,
\end{equation}
with $0<\theta\ll1$. As in \cite{bal2024complex}, we consider the weak coupling regime $0<\eps\ll1$ in two sub-regimes, the {\em kinetic} regime when $\eta=1$ and the {\em diffusive} regime when $\eta\ll1$. That the latter regime corresponds to long-distance propagation (optically thick media) may be obtained from the fact that distance of propagation times random strength (proportional to $\nu^2$) is $\eta/\eps \times \eps /\eta^3 = \eta^{-2} \gg1$. We will see below that $\eps$ models the width of the incident beam $u_0$. {\color{black} We also provide a discussion of these scalings in application at the end of Section~\ref{sec:main}.} For technical reasons as in \cite{bal2024complex}, we need a large scale separation between $\eps$ and $\eta$ in the diffusive regime and assume that $\eta=(\log|\log\eps|)^{-1}$. 

The results obtained in this paper hold in the kinetic regime so long as both $\eps\to0$ and $\theta\to0$. In the diffusive regime, we will need in addition $\theta e^{\frac C{\eta^2}}/\eta^2$ to be controlled and asymptotically small.  In order to simplify notation and not carry through the small parameters $(\theta,\eps,\eta)$, we assume that $\eps=\eps(\theta)$ as a function (continuous from $(0,\frac12]$ to $(0,\frac12]$) such that $\eps(\theta)\to0$ as $\theta\to0$. Therefore, $\eta=\eta(\theta)$ implicitly. We also assume for concreteness the existence of $0<\gamma$ so that
\begin{equation}\label{eq:csteps}
    0< \theta \leq \eps^{\gamma}(\theta),\qquad \eta(\theta):=(\log|\log\eps(\theta)|)^{-1}.
\end{equation}
We verify that $\theta^\alpha e^{\frac C{\eta^2}}$ is then bounded for any $\alpha>0$.

We now consider the standard paraxial envelope ansatz $u^\theta(z,x)=p\big(\frac{\eta z}{\eps\theta},x\big)e^{-\frac{i\eta k_0z}{\eps\theta^2}}$. After formally ignoring the backscattering term $(\frac{\eps\theta}{\eta})^2\partial^2_zu^\theta$ in \eqref{eqn:Helmholtz}, we arrive at the paraxial approximation
\begin{equation}\label{eq:paraxial}
    2ik_0\partial_zu^\theta+\frac{\eta}{\eps}\Delta_x u^\theta +\frac{k_0^2}{(\eta\eps\theta)^{\frac 12}}\nu\Big(\frac{\eta z}{\eps\theta},x\Big)u^\theta=0,\quad  z>0, x\in\mathbb{R}^d\,,
\end{equation}
with incident beam profile $u^\theta(0,x)=u^\theta_0(x)$. We will see below that the incident beams of our theory are broad, and more precisely functions of $\eps^\beta x$ for $\beta\geq1$. The justification of the paraxial regime from the full Helmholtz equation is difficult, with theoretical results obtained in \cite{bailly1996parabolic,garnier2009coupled}. The envelope equation \eqref{eq:paraxial} is our starting point. 

The objective of this paper is to show that in the limit $\theta\to0$, $u^\theta$ satisfies in law a limiting model consistent with complex Gaussianity and in the diffusive regime with a scintillation index equal to $1$; see Corollary \ref{cor:scint} below for notation. In this paper, we restrict ourselves to a setting of Gaussian fluctuations $\nu(z,x)$ to simplify the derivation.

The main step of our derivation is to show that the statistical moments of the paraxial solution $u^\theta$ in \eqref{eq:paraxial} are well approximated by the corresponding moments of the solution to the following It\^o-Schr\"odinger equation:
\begin{equation}\label{eq:uSDEd}
      \mathrm{d}u^\eps=\frac{i\eta}{2k_0\eps}\Delta_xu^\eps\mathrm{d}z-\frac{k_0^2R(0)}{8\eta^2}u^\eps\mathrm{d}z+\frac{ik_0}{2\eta}u^\eps\mathrm{d}B, \qquad u^\eps(0,x)=u_0^\eps(x).
\end{equation}
The above equation is deduced from \eqref{eq:paraxial} by formally replacing $\theta^{-\frac12}\nu(\frac{z}{\theta},x)\mathrm{d}z$ by its central limit approximation $\mathrm{d}B(z,x)$, where $B(z,x)$ is a mean-zero Gaussian process satisfying $\mathbb{E}[B(z,x)B(z',y)]=\min(z,z') R(x-y)$ with spatial covariance $R(x)$ defined in \eqref{eqn:R(x)} below. The product $u^\eps\mathrm{d}B$ should be interpreted in a standard It\^o form \cite{dawson1984random,fannjiang2004scaling}.
That the complex circular Gaussian conjecture holds for $u^\eps$ in the limit $\eps\to0$ was shown recently in \cite{bal2024complex}. Careful moment estimates of $u^\theta$ allow us to extend such a result to the paraxial model of wave propagation.

\medskip

The rest of the paper is structured as follows. The next Section \ref{sec:description} defines the statistical moments of interest in the derivation and describes our hypotheses on the (Gaussian) random medium. Our main results are collected in Section \ref{sec:main}. The proofs of the results are postponed to Section \ref{sec:proofs}.

We restrict ourselves to the setting of deterministic incident beams (fully coherent sources). The generalization in \cite{bal2024long} of the results in the It\^o-Schr\"odinger regime of \cite{bal2024complex} to the case of partially coherence sources can also presumably be extended to the paraxial regime.

\section{Description of the wavefield, random medium, and moments}\label{sec:description}

For simplicity, we set $2k_0=1$ since this (rescaled) parameter does not have any qualitative influence, and assume that the amplitude $\nu$ is appropriately rescaled so that the paraxial equation becomes
\begin{equation}\label{eqn:u_theta}
    \partial_zu^\theta=\frac{i\eta}{\eps}\Delta_x u^\theta+\frac{i}{(\eta\eps\theta)^{\frac 12}}\nu\Big(\frac{\eta z}{\eps\theta},x\Big)u^\theta.
\end{equation}
Define the (partial) Fourier transform of $u^\theta$ as
\begin{equation*}
\hat{u}^\theta(z,\xi)=\int\limits_{\mathbb{R}^d}u^\theta(z,x)e^{-i\xi\cdot x}\mathrm{d}x\,.
\end{equation*}
From \eqref{eqn:u_theta}, we deduce that
\begin{equation*}
    \partial_z\hat{u}^\theta=-\frac{i\eta}{\eps}|\xi|^2\hat{u}^\theta+\frac{i}{(\eta\eps\theta)^{\frac 12}}\int\limits_{\mathbb{R}^d}\hat{u}^\theta(z,\xi-k)\frac{{\color{black}\hat{\nu}(\frac{\eta z}{\eps\theta},\mathrm{d}k)}}{(2\pi)^d}, \quad \hat{u}^\theta(0,\xi)=\hat{u}^\theta_0(\xi)\,,
\end{equation*}
{\color{black} where we denote by $\hat{\nu}(z,\mathrm{d}k)$ the random measure such that $\nu(z,x)=\int_{\mathbb{R}^d}e^{ik\cdot x}\frac{\hat{\nu}(z,\mathrm{d}k)}{(2\pi)^d}$.} For a collection of points $(X,Y)=(x_1,\cdots,x_p,y_1,\cdots,y_q)\in\mathbb{R}^{(p+q)d}$, define the $p+q$th moment of the wavefield $u^\theta$ at fixed $z$ by
\begin{equation}\label{eq:mupq}
    \mu_{p,q}^\theta(z,X,Y)=\mathbb{E}[\prod\limits_{j=1}^pu^\theta(z,x_j)\prod\limits_{l=1}^qu^{\theta *}(z,y_l)]\,.
\end{equation}
Similarly, define the corresponding (complex symmetrized) Fourier transform by
\begin{equation*}
    \hat{\mu}_{p,q}^\theta(z,v)=\mathbb{E}[\prod\limits_{j=1}^p\hat{u}^\theta(z,\xi_j)\prod\limits_{l=1}^q\hat{u}^{\theta *}(z,\zeta_l)]\,,
\end{equation*}
where $v=(\xi_1,\cdots,\xi_p,\zeta_1,\cdots,\zeta_q)$ denotes the complex symmetrized (i.e., with inner product $\sum x_i\cdot\xi_i-\sum y_j\cdot\zeta_j$) vector of variables dual to $(X,Y)$. 

\paragraph{Gaussian random potential}

We assume that $\nu$ is a real valued Gaussian process stationary in $x$ and $z$ with covariance function
\begin{equation}\label{eq:covC}
    \mathbb{E}[\nu(z,x)\nu(z',x')]=\mathfrak{C}(z-z',x-x')\,.
\end{equation}
In particular, this implies that the Fourier transform of the covariance function with respect to $x$ and $x'$ is given by
\begin{equation*}
    (2\pi)^d\hat{\mathfrak{C}}(z-z',k)\delta(k+k')\,.
\end{equation*}
 We assume that $\hat{\mathfrak{C}}(z,\cdot)$ is integrable in $z$, smooth in $k$ and rapidly decaying for large values of $z$. For concreteness and to simplify some derivations, we assume that for $\langle k\rangle:=\sqrt{1+|k|^2}$, we have 
 \begin{equation}\label{eqn:C_n_assump}
{\color{black}\int_{\Rm^{d+1}}\aver{k}^{{d_0+3}}\aver{s}  |\hat{\mathfrak{C}}(s,k)| \mathrm{d}s\mathrm{d}k   = \lambda_1(d_0) <\infty}\,,
\end{equation}
where we assume that $d_0\geq d$. Except for improved regularity results obtained when $d_0\geq d$ in Lemma \ref{lemma:R1_theta_bound} (and Lemma \ref{lem:E_psi_theta_bounds}) and for better stochastic continuity properties in Theorem \ref{thm:tightness}, there is no loss in generality to assume that $d_0=d$, the lateral spatial dimension. 
We assume the standard symmetries for correlation functions: {\color{black}$\mathfrak{C}(z,x)=\mathfrak{C}(-z,-x)$} and denote by $R(x)$ the lateral spatial covariance
\begin{equation}\label{eqn:R(x)}
    R(x)=\int\limits_{-\infty}^\infty\mathfrak{C}(z,x)\mathrm{d}z\,.
\end{equation}
{\color{black}Note that the Fourier transform $\hat{R}(k)$ is the power spectral density of a stationary process and is non-negative from Bochner's theorem.} We next assume that $R\in \sL^1(\mathbb{R}^d)\cap \sL^\infty(\mathbb{R}^d)$ with the Fourier transform $\hat{R}(k)=\hat{R}(-k)\le \hat{\sfR}(k)= \hat{\sfR}(|k|)$ for some radially symmetric function $\hat{\sfR}$. 

We deduce from \eqref{eqn:C_n_assump} that $\langle k\rangle^2\hat{R}(k)\in \sL^1(\mathbb{R}^d)$, and when $d\ge 3$ further assume that $\langle k\rangle^{d-2}\hat{R}(k)\in\sL^\infty(\mathbb{R}^d)$. We finally assume that $R(x)$ is maximal at $x=0$ and that $\Xi=\nabla^2R(0)$ exists and is (strictly) negative definite. 

As in~\cite{fannjiang2004scaling}, we denote the potential in the rescaled coordinates by
\begin{equation}\label{eqn:nu_theta_def}
    \nu^\theta_z(\cdot)=\nu\Big(\frac{\eta z}{\eps\theta},\cdot\Big)\,,
\end{equation}
and the corresponding covariance as
  \begin{equation}\label{eqn:C_def}
            {\mathfrak{C}}^\theta(s,x)={\mathfrak{C}}\Big(\frac{\eta s}{\eps\theta},x\Big),
        \end{equation}
with the Fourier transform $\hat{\mathfrak{C}}^\theta(s,k)$ defined in a similar manner.
\paragraph{Assumptions on the source}
We assume a deterministic incident beam profile of the form
\begin{equation}\label{eq:u0eps}
    u_0^\theta(x)=u_0(\eps^\beta x)\,,
\end{equation}
with $u_0\in\mathcal{S}(\mathbb{R}^d)$ and $\beta\geq1$. This profile corresponds to a wide beam when $\eps\ll1$. The scaling $\beta=1$ is arguably the more interesting one. The case $\beta>1$ is a good model for an incident plane wave. Note that in the Fourier domain,
\begin{equation*}
    \hat{u}^\theta_0(k)=\eps^{-\beta d}\hat{u}_0\Big(\frac{k}{\eps^\beta}\Big)\,.
\end{equation*}
Due to this scaling, we have that for all $n\ge 0$,
\begin{equation*}
    \|\langle k\rangle^{n}\hat{u}^\theta_0(k)\|\le  \|\langle \eps^{-\beta}k\rangle^{n}\hat{u}^\theta_0(k)\|=\|\langle k\rangle^{n}\hat{u}_0(k)\|\le C\,
\end{equation*}
uniformly in $\eps\in(0,1]$. {\color{black}Here and below denote by $\mathcal{M}_B(\mathbb{R}^{m})$, the Banach space of finite signed measures on $\mathbb{R}^{m}$, equipped with the total variation (TV) norm $\|\cdot\|$. For measures with a density with respect to the Lebesgue measure (such as $\hat u^\theta_0(k)dk$ above with density $\hat u^\theta_0(k)$) then that norm is the {\color{black}$\sL^1(\mathbb{R}^{d})$} norm of that density. We use the notation $\|\cdot\|$ both for the TV norm of bounded Radon measures or for the operator norm of operators acting on such measures. The standard $\sL^p$ norms are denoted by $\|\cdot\|_p$.}

\paragraph{Spatially rescaled random vector}
As in \cite{bal2024complex}, we define the spatially rescaled random process $\phi^\theta$ as
\begin{equation}\label{eq:phitheta}
    \phi^\theta(z,r,x)=u^\theta(z,\eps^{-\beta}r+\eta x).
\end{equation}
It is for such a rescaled wavefield that we can expect a non-trivial limit as $\theta\to0$ and $\eps\to0$.

An element of the convergence result is the identification of finite dimensional distributions of $\phi^\theta$. Thus 
for $X=(x_1,\cdots,x_N)\in\mathbb{R}^{Nd}$, we also define a vector of such processes as
\begin{equation}\label{eq:Phitheta}
    \Phi^\theta(z,r,X)=(\phi^\theta(z,r,x_1),\cdots,\phi^\theta(z,r,x_N))\,.
\end{equation}
The moments of this random vector will be shown to be well approximated by that of a circularly symmetric Gaussian random vector in the limit $(\theta,\eps)\to 0$ (provided \eqref{eq:csteps} holds).

\section{Main results}\label{sec:main}

The limiting approximations for {\color{black}$\phi^\theta$} as $\theta\to0$ in the kinetic and diffusive regimes are the same as for the It\^o-Schr\"odinger model (as $\eps\to0$) treated in \cite{bal2024complex}. The limits are described by second-order moments, which we now recall.

In the {\em kinetic regime} where $\eta=1$, we define the second-order moment $M_{1,1}(z,r,x,y)$ as
\begin{equation}\label{eqn:M_11_kinetic}
    M_{1,1}(z,r,x,y)=\begin{cases}
        |u_0(r)|^2\mQ(y-x,0),\quad&\beta>1\\
        \dint\limits_{\mathbb{R}^d}e^{i\xi\cdot(r-r')}|u_0(r')|^2\mQ(y-x,\xi)\frac{\mathrm{d}\xi\mathrm{d}r'}{(2\pi)^d},\quad&\beta=1\,,
    \end{cases}
\end{equation}
where
\begin{equation}
    \mQ(\tau,\zeta)=\exp\Big(\frac{z}{\eta^2}\int\limits_{0}^1[R\big(\tau+2sz\zeta\big)-R(0)]\mathrm{d}s\Big)\,.
\end{equation}

In the {\em diffusive regime}, characterized by $\eta=(\log|\log\eps|)^{-1}$, we define $M_{1,1}(z,r,x,y)$ as
\begin{equation}\label{eqn:M_11_diffusive}
    M_{1,1}(z,r,x,y)=\begin{cases}
        |u_0(r)|^2e^{\frac{z}{2}(y-x)^\tps\Xi(y-x)},\quad &\beta>1\\[2mm]
        e^{\frac{1}{8}z(y-x)^\tps\Xi(y-x)}e^{-\frac{3i}{4z}(y-x)\cdot r}G(z^3,r)\ast[e^{\frac{3i}{4z}(y-x)\cdot r}|u_0(r)|^2],\quad&\beta=1\,,
    \end{cases}
\end{equation}
where $\Xi=\nabla^2R(0)$ (is negative definite) and $G(t,r)$ is the Green's function to the evolution equation
\begin{equation}\label{eq:Gdiff}
    \partial_tG+\frac{2}{3}\nabla_r\cdot(\Xi\nabla_r)G(t,r)=0,\quad G(0,r)=\delta_0(r)\,.
\end{equation}

We are now ready to state our main theorems. We assume \eqref{eq:csteps} throughout and recall that $\Phi^\theta$ is defined in \eqref{eq:Phitheta}.
\begin{theorem}\label{thm:complex_Gauss_kinetic}
In the kinetic regime $\eta=1$, $\Phi^\theta(z,r,X)\Rightarrow\Phi(z,r,X)$ in distribution as $\theta\to 0$ where $\Phi(z,r,X)$ is a random vector with elements $\{\phi_j\}_{j=1}^N$ satisfying $\mathbb{E}\phi_j=M_{1,0}(z,r,x_j)$ and ${\color{black}\mathcal{Z}}=\Phi-\mathbb{E}\Phi$ a mean zero circular complex Gaussian random vector with elements $\{{\color{black}\mathcal{Z}}_j\}_{j=1}^N$ satisfying
\begin{equation}
    \mathbb{E}[{\color{black}\mathcal{Z}}_j{\color{black}\mathcal{Z}}^\ast_l]=M_{1,1}(z,r,x_j,x_l)-M_{1,0}(z,r,x_j)M_{0,1}(z,r,x_l),\quad \mathbb{E}[{\color{black}\mathcal{Z}}_j{\color{black}\mathcal{Z}}_l]=0\,.
\end{equation}    
Here, $M_{0,1}(z,r,x)=M_{1,0}^\ast(z,r,x)=e^{-\frac{R(0)z}{2}}u_0(r)$ and $M_{1,1}(z,r,x,y)$ is given by~\eqref{eqn:M_11_kinetic}.
\end{theorem}

\begin{theorem}\label{thm:complex_Gauss_diff}
In the diffusive regime $\eta=(\log|\log\eps|)^{-1}$, $\Phi^\theta(z,r,X)\Rightarrow\Phi(z,r,X)$ in distribution as $\theta\to 0$ where $\Phi(z,r,X)$ is a mean zero circular complex Gaussian random vector with elements $\{\phi_j\}_{j=1}^N$ satisfying 
\begin{equation}
    \mathbb{E}[\phi_j\phi_l^\ast]=M_{1,1}(z,r,x_j,x_l),\quad   \mathbb{E}[\phi_j\phi_l]=0\,,
\end{equation}    
with $M_{1,1}(z,r,x,y)$ given by~\eqref{eqn:M_11_diffusive}.
\end{theorem}

As in \cite{bal2024complex}, the results in Theorems \ref{thm:complex_Gauss_kinetic} and \ref{thm:complex_Gauss_diff} at $(z,r,X)$ fixed characterize finite dimensional distributions of the process $x\to\phi^\theta(z,r,x)$ defined in \eqref{eq:phitheta} at $(z,r)$ fixed. The proof of convergence of the whole process follows from a tightness and stochastic continuity result of independent interest, which we now state.
\begin{theorem}\label{thm:tightness}
    \begin{equation}
        \sup\limits_{0\le s\le z}\mathbb{E}|\phi^\theta(s,r,x+h)-\phi^\theta(s,r,x)|^{2n}\le C(n,\alpha,z){\color{black}\|h\|_2}^{2n\alpha}
    \end{equation}
for $\|h\|_2<1$ uniformly in $\theta\in (0,\frac12]$. Here, {\color{black} $\alpha$ is any scalar such that} $\alpha\in(0,1]$ in the kinetic regime and $\alpha\in(0,1)$ in the diffusive regime. Let $n$ be such that $2n\alpha\ge d+2n\alpha_{-}$,  $\alpha_{-}\in(0,\alpha)$. Then the process $\phi^{\theta}(z,r)$ is tight on $C^{0,\alpha_{-}}(\mathbb{R}^d)$  and there exists a  H\"{o}lder continuous version of $\phi^\theta(z,r)$ on $C^{0,\alpha_{-}}(\mathbb{R}^d)$. 
\end{theorem}
All results are proved in Section \ref{sec:proofs}. 

As in \cite[Corollary 2.8]{bal2024complex}, we deduce the following result of convergence in distribution:
\begin{corollary}[Convergence of processes]\label{cor:conv}
    For $z$ and $r$ fixed and under the hypotheses of Theorem \ref{thm:tightness}, the process $\phi^\theta(z,r,\cdot)$ converges in distribution as probability measures on $C^{0,\alpha_-}(\Rm^d)$ to its unique {\color{black}(Gaussian) limit. This} is characterized by limits of finite dimensional distributions shown to converge in Theorems \ref{thm:complex_Gauss_kinetic} and \ref{thm:complex_Gauss_diff} in the kinetic and diffusive regimes. 
\end{corollary}

For completeness, we recall the following corollary on the scintillation index. This is \cite[Corollary 2.5]{bal2024complex}, written with $\beta=1$ to simplify the presentation, and to which we refer for derivation and details.
\begin{corollary}[Scintillation in the diffusive regime]\label{cor:scint}
Assume $\beta=1$ and define the intensity
\begin{equation}\label{eqn:I_eps}
    I^\theta(z,r,x)= |u^\theta(z,\eps^{-1} r +\eta x)|^2\,.
\end{equation}
    Then (at fixed $(z,r)$) $I^\theta\boldsymbol{\Rightarrow}I$ in distribution as $\theta\to 0$ where $I(z,r)$ is distributed according to an exponential law with $\mathbb{E}[I](z,r)=(G(z^3,\cdot)*|u_0|^2)(z,r)$, where $G$ is the fundamental solution to the diffusion equation \eqref{eq:Gdiff} and $*$ denotes convolution (in the $r-$variable). In particular, the scintillation index
    \begin{equation}\label{eq:scintindex}
    {\mathsf S}(z,r):=\frac{\mathbb{E}[I^2]-\mathbb{E}[I]^2}{\mathbb{E}[I]^2}=1,\quad\forall z>0, \ r\in\mathbb{R}^d\,.
    \end{equation}
\end{corollary}
The above corollary fully characterizes the scintillation regime with a scintillation index uniformly equal to $1$ corresponding to a fully developed speckle pattern. The above diffusion with kernel $G(z^3,\cdot)$ is naturally written in the `time' variable $z^3$ and lateral `spatial' variables $r$, resulting in a very anomalous diffusion regime with $r\sim z^{\frac32}$ for large $z$ reflecting similar beam spreading in the paraxial and It\^o-Schr\"odinger regimes.

{
\color{black} We also note the independence of the macroscopic wavefields at different $r$ under the diffusive scaling in the following theorem. \begin{theorem}\label{thm:phi_r_indep}
    For fixed z and $r\neq r'$ in the diffusive regime, the random vectors $\Phi^\theta(z,r,X)$ and $\Phi^\theta(z,r',X')$ become statistically independent as $\theta\to 0$.
\end{theorem}
}

\paragraph{Scaling in applications} We briefly revisit the scaling \eqref{eq:scalings} and its practical relevance{\color{black}, say for laser beam propagation through turbulence}. There, $x$ was measured in terms of the smallest scale of turbulence, let us call it $l_0$, typically taken in practice to be in the $mm$ range. This means that medium fluctuations should be understood as $\nu=\nu\big(\frac{z}{l_0},\frac{x}{l_0}\big)$. With $Z$ a dimension-less parameter characterizing longitudinal distance of propagation, $w_0$ a parameter describing the width of the source, and $\sigma$ a covariance parameter for the random fluctuations $\nu=\sigma \tilde\nu$, we have the scaling and paraxial model after neglecting backscattering:
\[
\frac{z}{l_0Z}\to z,\quad \frac{x}{l_0}\to x,\qquad2i\Big(\frac{k_0l_0}{Z}\Big)\partial_zu+\Delta_xu+k_0^2l_0^2\sigma\tilde{\nu}(zZ,x)u=0,\quad u(0,x)=u_0\Big(\frac{l_0x}{w_0}\Big).
\]
Consider the scaling with $\beta=1$ for concreteness. The parameters $(\theta,\eps,\eta)$ are then defined as
\[
\theta=\frac{1}{k_0l_0},\quad \eps=  \frac{l_0}{w_0},\quad \eta = \frac{\eps Z}{k_0l_0}  = \frac{Z}{k_0w_0}.
\]
We then verify that the above paraxial model gives \eqref{eq:paraxial} when $\sigma^2Z^3 l_0^2=w_0^2$.  This is equivalent to $(\sigma^2l_0^2k_0^2) \eta^2 Z=1$.

{\color{black}A typical wavelength used in practice is $\lambda\approx 0.6\times 10^{-6}m$ and corresponds to $k_0\approx10^7 m^{-1}$. Under weak turbulence assumptions \cite{clifford2005classical, ishimaru1978wave, andrews2001laser}, the refractive index is approximated as $n(\Vec{r})\approx 1+ 79\times 10^{-6}\frac{P(\Vec{r})}{T(\Vec{ r})}$, where $P$ is the pressure in millibars and $T$ the temperature in Kelvin. Assuming the variations in the refractive index are mainly from temperature fluctuations ($T(\Vec{r})=T+\delta T(\Vec{r})$), this corresponds to $\delta n\approx 79\times 10^{-6}\frac{P}{T^2}\delta T$. So for changes in the temperature in the $0.1^\circ-1^\circ$ range, $\sigma=(\mathbb{E}[\delta n^2])^{\frac 12}\approx10^{-7}-10^{-6}$. We also set the fluctuation lengthscale $l_0=2\times 10^{-3}m$ so that $\theta=5\times 10^{-5}$.

We thus observe that we reach the kinetic regime $\eta=1$ for a propagation distance of order $l_0Z=500m$ when $\sigma=10^{-7}$ with $\eps=\frac{\eta k_0l_0}{Z}=0.08\ll1$ when $w_0\ge  \frac{10^{-3}}{\eps}$, which is the case for the beam widths of a (quite a) few millimeters to meters found in applications \cite{nelson2016scintillation, vorontsov2020atmospheric, stotts2021adaptive, andrews2001laser}. 

The diffusive regime is reached for $\eta\ll1$, corresponding to $\sigma\approx5\times 10^{-7}$ and $l_0Z$ in the range $10^3m$ typical in applications. For $l_0Z=10^3m$ so that $\eta=0.1$ and $\eps=10^{-2}$ for concreteness, the proposed theory has a domain of validity for beam width $w_0=10^{-3}/\eps$, i.e., for beams with a diameter in tens of centimeters to meter scale.}


\section{Proofs of the main results} \label{sec:proofs}

The main strategy of the proof of the above results is to show that the exact moment equations available in the It\^o-Schr\"{o}dinger regime are approximately valid in the paraxial regime. Once such a result is established, the characterization of the limits of finite dimensional distributions of $\phi^\theta$ closely follows the derivation in \cite{bal2024complex}. 

\subsection{Phase compensated wave field and preparatory calculations}

As in \cite{bal2011asymptotics,gu2021gaussian} and \cite{bal2024complex}, approximations to the wavefield are best performed in the Fourier domain and after `removing' or `compensating for' the effect of free propagation in a homogeneous medium. 
The phase compensated wave field $\Tilde{u}^\theta(z,\xi)=\hat{u}^\theta(z,\xi)e^{\frac{i\eta z}{\eps}|\xi|^2}$ solves the evolution equation
\begin{equation}\label{eqn:u_tilde}
     \partial_z\tilde{u}^\theta=\frac{i}{(\eta\eps\theta)^{\frac 12}}\int\limits_{\mathbb{R}^d}\tilde{u}^\theta(z,\xi-k)e^{\frac{i\eta z}{\eps}g(\xi,k)}\frac{{\color{black}\hat{\nu}(\frac{\eta z}{\eps\theta},\mathrm{d}k)}}{(2\pi)^d}, \quad \tilde{u}^\theta(0,\xi)=\hat{u}^\theta_0(\xi)\,,
\end{equation}
where $g(\xi,k)=|\xi|^2-|\xi-k|^2=-|k|^2+2k\cdot\xi$. In terms of the phase compensated wave field, the $p+q$th moment $\hat{\mu}^\theta_{p,q}$ is then given by
\[
\hat{\mu}_{p,q}^\theta(z,v)=\Pi_{p,q}^\theta(z,v)\mathbb{E}[\psi^\theta_{p,q}(z,v)],
\]
where 
\[\Pi_{p,q}^\theta(z,v)=e^{-\frac{i\eta z}{\eps}(\sum_{j=1}^p|\xi_j|^2-\sum_{l=1}^q|\zeta_l|^2)},\]
and $\psi^\theta_{p,q}$ is a product of $p+q$ copies of the phase compensated wave field given by
\[\psi^\theta_{p,q}(z,v)=\prod\limits_{j=1}^p\tilde{u}^\theta(z,\xi_j)\prod\limits_{l=1}^q\tilde{u}^{\theta *}(z,\zeta_l).\]
Denote by $\Psi^\theta_{p,q}$ the $p+q$th moment of the phase compensated field:
\begin{equation}\label{eq:Psieps}
\Psi_{p,q}^\theta(z,v) =\mathbb{E}\psi_{p,q}^\theta(z,v).
\end{equation}
We drop the $p,q$ dependence on $\psi$ for notational convenience. From~\eqref{eqn:u_tilde}, we have
\begin{equation}\label{eqn:psi_theta}
    \partial_z\psi^\theta=\frac{i}{(\eta\eps\theta)^{\frac 12}}L^\theta(z,\nu^\theta_z)\psi^\theta,\quad \psi^\theta(0,v)=\prod\limits_{j=1}^p\hat{u}^\theta_0(\xi_j)\prod\limits_{l=1}^q\hat{u}^{\theta *}_0(\zeta_l)\,,
\end{equation}
where we recall that $\nu^\theta_z(\cdot)=\nu\big(\frac{\eta z}{\eps\theta},\cdot\big)$ and the operator $L^\theta$ is given by
\begin{equation}\label{eqn:L_theta}
L^\theta(s,\nu_z)\psi=\int\limits_{\mathbb{R}^d}\Big[\sum\limits_{j=1}^pe^{\frac{i\eta s}{\eps}g(\xi_j,k)}\psi(\xi_j-k)-\sum\limits_{l=1}^qe^{-\frac{i\eta s}{\eps}g(\zeta_l,-k)}\psi(\zeta_l+k)\Big]\frac{{\color{black}\hat{\nu}(z,\mathrm{d}k)}}{(2\pi)^d}\,,
\end{equation}
where we explicitly indicate only the variables in $\psi$ that are shifted by $\pm k$.  
\begin{lemma}\label{lemma:chaos}
    $\psi^\theta$ can be expanded as
    \begin{equation}\label{eqn:chaos_series}
    \begin{aligned}
        \psi^\theta(z,v)=\sum\limits_{n\ge 0}\frac{{\color{black}1}}{(\eta\eps\theta)^{\frac{n}{2}}}\psi^\theta_n(z,v)\,,
    \end{aligned}
\end{equation}
where $\psi^\theta_0(z,v)=\psi^\theta(0,v)$ and for $n\ge 1$,
\[ \psi^\theta_n(z,v)=\int\limits_{[0,z]^n_<}\Big(\prod\limits_{j=1}^nL^\theta(s_j,\nu^\theta_{s_j})\Big)\psi^\theta(0,v)\prod\limits_{j=1}^n\mathrm{d}s_j,\]
where $[0,z]^n_<$ denotes the simplex $0\le s_1\le\cdots\le s_n\le z$.  Moreover, 
\begin{equation}\label{eqn:chaos_2N}
    \begin{aligned}
        \mathbb{E}\psi^\theta_{2N}(z,v)=\int\limits_{[0,z]^{2N}_<}\int\limits_{\mathbb{R}^{Nd}}\sum\limits_{P^{2N}}\sum\limits_{j=1}^{(p+q)^{2N}}e^{\frac{i\eta}{\eps}G'_j(\Vec{s},\Vec{k},v)}\psi^\theta(0,v-A'_j\Vec{k})\prod\limits_{(j',l')\in P^{2N}}\hat{\mathfrak{C}}^\theta(s_{j'}-s_{l'},k_{j'})\frac{\mathrm{d}k_{j'}}{(2\pi)^d}\prod\limits_{j=1}^{2N}\mathrm{d}s_j\,.
    \end{aligned}
\end{equation}
Here, $A_j'$ is a $(p+q)d\times Nd$ block diagonal matrix with uniformly bounded entries and $G_j'$ is real valued $\forall j$. {\color{black}$P^{2N}$ denotes a permutation of pairings of $2N$ indices.}
\begin{proof}
Using \eqref{eqn:psi_theta} and a standard Duhamel expansion, we arrive at~\eqref{eqn:chaos_series}. Defining these iterated integrals and the convergence in \eqref{eqn:chaos_series} in a square integrable sense, can be justified as in, e.g., the proof of \cite[Proposition 2.1]{bal2011asymptotics}. Using the definition of $L^\theta$, this can be further expanded as
\begin{equation}\label{eqn:chaos_series2}
    \begin{aligned}        
    \psi^\theta_n(z,v) =\int\limits_{[0,z]^n_<}\int\limits_{\mathbb{R}^{nd}}\sum\limits_{j=1}^{(p+q)^n}e^{\frac{i\eta}{\eps}G_j(\Vec{s},\Vec{k},v)}\psi^\theta(0,v-A_j\Vec{k})\prod\limits_{j=1}^n{\color{black}\frac{\hat{\nu}^\theta_{s_j}(\mathrm{d}k_j)}{(2\pi)^d}\mathrm{d}s_j}\,,
    \end{aligned}
\end{equation}
where $\Vec{s}=(s_1,\cdots,s_n)$ and $\Vec{k}=(k_1,\cdots,k_n)$. Here, $A_j$ are matrices and $G_j$ are real valued functions which we now define. 

 $A_j$ are $(p+q)d\times nd$ block matrices with $d\times d$ blocks such that $A_j$ has exactly one non zero entry per column block. This non zero entry equals 1 when it appears in the first $pd$ rows and equals $-1$ otherwise. $G_j$ can be defined recursively as follows:

Let $\epsilon_m=1$ if $m=1,\cdots,p$ and equal to $-1$ if $m=p+1,\cdots,p+q$. Also, let
\begin{equation*}
G_{1,m}(\Vec{s},\Vec{k},v)=s_1\epsilon_mg(v_m,\epsilon_mk_1) {\color{black}+ \epsilon_m\frac{\pi}{2}\frac{\eps}{\eta}}\,.    
\end{equation*}
Now, for $r\ge 1$, let $G_{r+1,j}(\Vec{s},\Vec{k},v)$ be given by
\begin{equation*}
\begin{aligned}
G_{r+1,(p+q)(m'-1)+m}(\Vec{s},\Vec{k},v)&= s_{r+1}\epsilon_mg(v_m,\epsilon_m k_{r+1}){\color{black}+ \epsilon_m\frac{\pi}{2}\frac{\eps}{\eta}}\\
&+G_{r,m'}(s_1,\cdots,s_r,k_1,\cdots,k_r,v_1,\cdots,v_m-\epsilon_mk_{r+1},v_{m+1},\cdots,v_{p+q})\,,    
\end{aligned}
\end{equation*}
where $m'=1,\cdots,(p+q)^r $ and $m=1,\cdots,p+q$. Here, for $r$ fixed, $j$ varies from $1$ to $(p+q)^r$. Finally set $G_j$ to be $G_{n,j}$.

From the Gaussian assumption on $\nu$, chaos with odd $n$ vanish after taking expectation. Applying Wick's theorem~\cite{janson1997gaussian}, we have 
\begin{equation*}
    \begin{aligned}
        \mathbb{E}\psi^\theta_{2N}(z,v)=\int\limits_{[0,z]^{2N}_<}\int\limits_{\mathbb{R}^{2Nd}}\sum\limits_{P^{2N}}\prod\limits_{(j',l')\in P^{2N}}(2\pi)^d\hat{\mathfrak{C}}^\theta(s_{j'}-s_{l'},k_{j'})\delta(k_{j'}+k_{l'})
        \\ \times \sum\limits_{j=1}^{(p+q)^{2N}}e^{\frac{i\eta}{\eps}G_j(\Vec{s},\Vec{k},v)}\psi^\theta(0,v-A_j\Vec{k})\prod\limits_{j=1}^{2N}\frac{\mathrm{d}k_j\mathrm{d}s_j}{(2\pi)^d}\,.
    \end{aligned}
\end{equation*}
The Dirac measure reduces integrals over $\Vec{k}\in\mathbb{R}^{2Nd}$ to an integral over $\mathbb{R}^{Nd}$. This gives~\eqref{eqn:chaos_2N}.

In~\eqref{eqn:chaos_2N}, $A_j'$ are $(p+q)d\times Nd$ block diagonal matrices of $d\times d$ blocks with either no non zero entries or exactly two non zero entries (either 1 or -1) per column block. Also, for a fixed arrangement $P^{2N}$ with $(l,m)\in P^{2N}$, $G_j'(\Vec{s},\Vec{k},v)=G_j(\Vec{s},B_{P^{2N}}(\Vec{k}),v)$ where $B_{P^{2N}}(\Vec{k})$ is a block matrix of size $2Nd\times 1$ with  $k_1,\cdots,k_N$ appearing as $l$th entries and $-k_1,-k_2,\cdots,-k_N$ appearing as corresponding $m$th entries. For example, when $N=2$, there are three possibilities for $P^4$: if $P^{4}=\{(1,2), (3,4)\}$, $B_{P^4}(\Vec{k})=(k_1,-k_1,k_2,-k_2)^{\top}$ and if $P^4=\{(1,3),(2,4)\}$, $B_{P^4}(\Vec{k})=(k_1,k_2,-k_1,-k_2)^\top$ and finally if $P^4=\{(1,4),(2,3)\}$, $B_{P^4}(\Vec{k})=(k_1,k_2,-k_2,-k_1)^\top$. 
\end{proof}
\end{lemma}

The analysis of $\Psi^\theta_{p,q}$ in \eqref{eq:Psieps} could also be carried out by computing products of Duhamel expansions of $\psi^\theta$ (with $p=1$ and $q=0$). It proved more convenient to write a Duhamel expansion directly for $\psi^\theta_{p,q}$ in \eqref{eqn:chaos_series} as it is more closely related to the moment equations available for the It\^o-Schr\"odinger solution and used in the analysis in \cite{bal2024complex}. For the rest of the paper, all we need of the above computations is the form of the moments in \eqref{eqn:chaos_2N} and the fact that the entries of all matrices $A'_j$ are bounded and that the phase terms $G_j'$ are real-valued. This allows us to obtain the following moment estimates.

\begin{lemma}\label{lem:E_psi_theta_bounds}
We have the following estimates:
\begin{equation}\label{eqn:E_psi_theta_bound1}
    \|\mathbb{E}\psi^\theta(z)\|\le \|\psi^\theta(0)\|e^{\frac{(p+q)^2z}{2\eta^2}{\color{black}\int_{\mathbb{R}}\|\hat{C}(s,\cdot)\|\mathrm{d}s}}\,,
\end{equation}
and
\begin{equation}\label{eqn:E_psi_theta_bound2}
    \|\sum\limits_{\substack{n,n'=1\\ n\neq n'}}^{p+q}\langle v_n\rangle^{\alpha'}\langle v_{n'}\rangle^{\alpha'}\Big(\prod\limits_{m=1}^{p+q}\langle v_m\rangle^\alpha\Big)\mathbb{E}\psi^\theta(z)\|\le \|\prod\limits_{m=1}^{p+q}\langle v_m\rangle^{\alpha+\alpha'}\psi^\theta(0)\|e^{\frac{\lambda[{\color{black}C},p,q,z,\alpha]}{\eta^2}} \,,
\end{equation}
with $\lambda$ defined in \eqref{eqn:lambda_def} below for $\alpha'\in[0,1]$ and $0\le\alpha\le\frac{d_0+1}{p+q}$. 
\end{lemma}
\begin{proof}
    From the chaos expansion~\eqref{eqn:chaos_series},~\eqref{eqn:chaos_2N},  we have
    \begin{equation*}
        \begin{aligned}
            \|\mathbb{E}\psi^\theta(z)\|\le \sum\limits_{N\ge 0}\frac{1}{(\eta\eps\theta)^N}\int\limits_{[0,z]^{2N}_<}\int\limits_{\mathbb{R}^{Nd}}\sum\limits_{P^{2N}}\sum\limits_{j=1}^{(p+q)^{2N}}\|\psi^\theta(0,v-A_j'\Vec{k})\|\prod\limits_{(j',l')\in P^{2N}}{\color{black}|\hat{\mathfrak{C}}^\theta(s_{j'}-s_{l'},k_{j'})|}\frac{\mathrm{d}k_{j'}}{(2\pi)^d}\prod\limits_{j=1}^{2N}\mathrm{d}s_j\,.
        \end{aligned}
    \end{equation*}
    Since a shift in $v$ does not affect the total variation norm, we have that
       \begin{equation*}
        \begin{aligned}
            \|\mathbb{E}\psi^\theta(z)\|&\le \|\psi^\theta(0)\| \sum\limits_{N\ge 0}\frac{(p+q)^{2N}}{(\eta\eps\theta)^N}\int\limits_{[0,z]^{2N}_<}\int\limits_{\mathbb{R}^{Nd}}\sum\limits_{P^{2N}}\prod\limits_{(j',l')\in P^{2N}}|{\color{black}\hat{\mathfrak{C}}^\theta(s_{j'}-s_{l'},k_{j'})|}\frac{\mathrm{d}k_{j'}}{(2\pi)^d}\prod\limits_{j=1}^{2N}\mathrm{d}s_j\\
            &= \|\psi^\theta(0)\| \sum\limits_{N\ge 0}\frac{(p+q)^{2N}}{(\eta\eps\theta)^N}\int\limits_{[0,z]^{2N}_<}\sum\limits_{P^{2N}}\prod\limits_{(j',l')\in P^{2N}}{\color{black}\|\hat{\mathfrak{C}}^\theta(s_{j'}-s_{l'},\cdot)\|}\prod\limits_{j=1}^{2N}\mathrm{d}s_j\,.
        \end{aligned}
    \end{equation*}
    Now we can symmetrize the simplex:
    \begin{equation*}
        \begin{aligned}
            \|\mathbb{E}\psi^\theta(z)\| &\le \|\psi^\theta(0)\| \sum\limits_{N\ge 0}\frac{1}{(2N)!}\frac{(p+q)^{2N}}{(\eta\eps\theta)^N}\int\limits_{[0,z]^{2N}}\sum\limits_{P^{2N}}\prod\limits_{(j',l')\in P^{2N}}{\color{black}\|\hat{\mathfrak{C}}^\theta(s_{j'}-s_{l'},\cdot)\|}\prod\limits_{j=1}^{2N}\mathrm{d}s_j\\
            &=\|\psi^\theta(0)\| \sum\limits_{N\ge 0}\frac{(p+q)^{2N}}{2^NN!}\Big(\frac{1}{\eta\eps\theta}\int\limits_{[0,z]^2}{\color{black}\|\hat{\mathfrak{C}}^\theta
            (s_1-s_2,\cdot)\|}\mathrm{d}s_1\mathrm{d}s_2\Big)^N\,.
        \end{aligned}
    \end{equation*}
  Here we have used the fact that $|P^{2N}|=\frac{(2N)!}{2^NN!}$.  From a change of variables $\frac{\eta}{\eps\theta}(s_1-s_2)\to s_2$, we have
    \begin{equation*}
        \begin{aligned}
            \|\mathbb{E}\psi^\theta(z)\| &\le \|\psi^\theta(0)\| \sum\limits_{N\ge 0}\frac{(p+q)^{2N}}{2^NN!}\Big(\frac{z}{\eta^2}{\color{black}\int_{\mathbb{R}}\|\hat{C}(s,\cdot)\|\mathrm{d}s}\Big)^N= \|\psi^\theta(0)\| e^{\frac{(p+q)^2z}{2\eta^2}{\color{black}\int_{\mathbb{R}}\|\hat{C}(s,\cdot)\|\mathrm{d}s}}\,.
        \end{aligned}
    \end{equation*}
For the second part, we note that
 \begin{equation*}
     \begin{aligned}
\|\sum\limits_{\substack{n,n'=1\\ n\neq n'}}^{p+q}\prod\limits_{m=1}^{p+q}\langle v_m\rangle^{\alpha_m}\mathbb{E}\psi^\theta(z)\|&\le \sum\limits_{N\ge 0}\frac{1}{(\eta\eps\theta)^N}\int\limits_{[0,z]^{2N}_<}\int\limits_{\mathbb{R}^{Nd}}\sum\limits_{P^{2N}}\sum\limits_{j=1}^{(p+q)^{2N}}\|\sum\limits_{\substack{n,n'=1\\ n\neq n'}}^{p+q}\prod\limits_{m=1}^{p+q}\langle v_m\rangle^{\alpha_m}\psi^\theta(0,v-A_j'\Vec{k})\|\\
&\times\prod\limits_{(j',l')\in P^{2N}}{\color{black}|\hat{\mathfrak{C}}^\theta(s_{j'}-s_{l'},k_{j'})|}\frac{\mathrm{d}k_{j'}}{(2\pi)^d}\prod\limits_{j=1}^{2N}\mathrm{d}s_j   \,,
     \end{aligned}
 \end{equation*}
 where $\alpha_m=\alpha$ if $m\neq n,n'$ and $\alpha_m=\alpha'+\alpha$ if $m=n,n'$.
 
 For each arrangement $A_j'$, we make a change of variables $v_m-[A'_j\Vec{k}]_m$ so that
 \begin{equation}\label{eqn:v_m_alpha_prod1}
     \begin{aligned}
      \|\prod\limits_{m=1}^{p+q}\langle v_m\rangle^{\alpha_m}\psi^\theta(0,v-A_j'\Vec{k})\|&=\|\prod\limits_{m=1}^{p+q}\langle v_m+[A'_j\Vec{k}]_m\rangle^{\alpha_m}\psi^\theta(0,v)\|\,.
     \end{aligned}
 \end{equation}
 We note that
 \begin{equation*}
   1\le  \langle \xi+k\rangle^2=1+|\xi+k|^2\le 1+2(|\xi|^2+|k|^2)\le 2\langle \xi\rangle^2\langle k\rangle^2\,.
 \end{equation*}
Using Lemma~\ref{lemma:chaos},  this can be generalized to give
 \begin{equation}\label{eqn:v_m_alpha_prod2}
  1\le   \langle v_m+[A'_j\Vec{k}]_m\rangle^{\alpha_m}\le C^N\langle v_m\rangle^{\alpha_m}\prod\limits_{j=1}^N\langle [\Vec{k}]_j\rangle^{\alpha_m}\,.
 \end{equation}
 This gives
 \begin{equation*}
     \begin{aligned}
 \|\sum\limits_{\substack{n,n'=1\\ n\neq n'}}^{p+q}\prod\limits_{m=1}^{p+q}\langle v_m\rangle^{\alpha_m}\mathbb{E}\psi^\theta(z)\|&\le \|\prod\limits_{m=1}^{p+q}\langle v_m\rangle^{\alpha'+\alpha}\psi^\theta(0)\|\sum\limits_{N\ge 0}\frac{C^{(p+q)N}(p+q)^{2N}}{(\eta\eps\theta)^N}\\
 &\times\int\limits_{[0,z]^{2N}_<}\int\limits_{\mathbb{R}^{Nd}}\sum\limits_{P^{2N}}\prod\limits_{(j',l')\in P^{2N}}{\color{black}|\hat{\mathfrak{C}}^\theta(s_{j'}-s_{l'},k_{j'})|}\langle k_{j'}\rangle^{2+(p+q)\alpha}\frac{\mathrm{d}k_{j'}}{(2\pi)^d}\prod\limits_{j=1}^{2N}\mathrm{d}s_j   \,.           
     \end{aligned}
 \end{equation*}
 Then as before, symmetrizing the simplex and making an appropriate change of variables gives
 \begin{equation*}
     \begin{aligned}
 &\|\sum\limits_{\substack{n,n'=1\\ n\neq n'}}^{p+q}\prod\limits_{m=1}^{p+q}\langle v_m\rangle^{\alpha_m}\mathbb{E}\psi^\theta(z)\|\\
 &\le \|\prod\limits_{m=1}^{p+q}\langle v_m\rangle^{\alpha'+\alpha}\psi^\theta(0)\|\sum\limits_{N\ge 0}\frac{C^{(p+q)N}(p+q)^{2N}}{2^NN!\eta^{2N}}z^N{\color{black}\Big(\int_{\mathbb{R}^{d+1}}\langle k\rangle^{2+(p+q)\alpha}|\hat{C}(s,k)|\mathrm{d}s\mathrm{d}k\Big)}^N \\
 &\le \|\prod\limits_{m=1}^{p+q}\langle v_m\rangle^{\alpha'+\alpha}\psi^\theta(0)\|e^{\frac{\lambda[{\color{black}C},p,q,z,\alpha]}{\eta^2}},           
     \end{aligned}
 \end{equation*}
 with
 \begin{equation}\label{eqn:lambda_def}
     \lambda[{\color{black}C},p,q,z,\alpha]=\frac{1}{2}C^{p+q}(p+q)^2z{\color{black}\int_{\mathbb{R}^{d+1}}\langle s\rangle\langle k\rangle^{2+(p+q)\alpha}|\hat{C}(s,k)|\mathrm{d}s\mathrm{d}k}\,.
 \end{equation}
\end{proof}
\subsection{Convergence of moments: from paraxial to It\^o-Schr\"{o}dinger}

The main results in Theorems \ref{thm:complex_Gauss_kinetic} and \ref{thm:complex_Gauss_diff} rely on the following decomposition, which is the main result of this paper:
\begin{theorem}\label{thm:psi_theta_decomp}
    The moments of the phase compensated wave field in \eqref{eq:Psieps} may be decomposed as
    \begin{equation*}
        \Psi_{p,q}^\theta(z,v) =\mathcal{B}^\theta_{p,q}(z,v)+\mathscr{R}^\theta_{p,q}(z,v)\,,
    \end{equation*}
    where $\mathcal{B}^\theta_{p,q}(z,v)$ is the solution to 
    \begin{equation}\label{eqn:B_theta}
        \mathcal{B}^\theta_{p,q}(z,v)=\psi^\theta(0,v)+\int\limits_{0}^z\mathcal{L}^\theta(s)\mathcal{B}^\theta_{p,q}(s)\mathrm{d}s\,,
    \end{equation}
    with $\mathcal{L}^\theta$ given by
    \begin{equation}\label{eqn:L_limit_expnsn}
        \begin{aligned}
            \mathcal{L}^\theta(z)\psi&=\frac{1}{\eta^2}\int\limits_{\mathbb{R}^d}\hat{R}(k)\Big[\sum\limits_{j=1}^p\sum\limits_{l=1}^qe^{\frac{i\eta z}{\eps}[g(\xi_j,k)-g(\zeta_l,k)]}\psi(\xi_j-k,\zeta_l-k)\\
            &-\sum\limits_{1\le j<j'\le p}e^{\frac{i\eta z}{\eps}[g(\xi_j,k)+g(\xi_{j'},-k)]}\psi(\xi_j-k,\xi_{j'}+k)\\
            &-\sum\limits_{1\le l<l'\le q}e^{-\frac{i\eta z}{\eps}[g(\zeta_l,k)+g(\zeta_{l'},-k)]}\psi(\zeta_l-k,\zeta_{l'}+k)-\frac{p+q}{2}\psi\Big]\frac{\mathrm{d}k}{(2\pi)^d}\,,
        \end{aligned}
    \end{equation}
    and $\mathscr{R}^\theta_{p,q}(z,v)$ satisfying 
    \begin{equation}\label{eqn:R_theta_bound}
        \|\prod\limits_{m=1}^{p+q}\langle v_m\rangle^\alpha \mathscr{R}^\theta_{p,q}(z,v)\|\le c(p,q,z)\frac{\theta}{\eta^2}e^{\frac{\lambda[{\color{black}C},p,q,z,\alpha]}{\eta^2}}\|\prod\limits_{m=1}^{p+q}\langle v_m\rangle^2\psi^\theta(0,v)\|\,,
    \end{equation}
    where $\lambda[{\color{black}C},p,q,z,\alpha]=\frac{1}{2}C^{p+q}(p+q)^2z{\color{black}\int_{\mathbb{R}^{d+1}}\langle s\rangle\langle k\rangle^{2+(p+q)\alpha}|\hat{C}(s,k)|\mathrm{d}s\mathrm{d}k}$ for a fixed constant $C$ and $0\le\alpha\le\frac{d_0+1}{p+q}$. \end{theorem}
\begin{proof}
        Iterating~\eqref{eqn:psi_theta} twice in integral form gives
\begin{equation}\label{eqn:E_psi_exact}
    \mathbb{E}\psi^\theta(z,v) =\mathbb{E}\psi^\theta(0,v)-\frac{1}{\eta\eps\theta}\int\limits_{0}^z\int\limits_{0}^s\mathbb{E}[L^\theta(s,\nu^\theta_s)L^\theta(s',\nu^\theta_{s'})\psi^\theta(s')]\mathrm{d}s'\mathrm{d}s\,.
\end{equation}
We use the following relations in turn 
\begin{align*}
    \psi^\theta & = \E \psi^\theta + (\psi^\theta-\E\psi^\theta)\\
    L^\theta(s',\nu^\theta_{s'}) & =L^\theta(s,\nu^\theta_{s'}) + (L^\theta(s',\nu^\theta_{s'})-L^\theta(s,\nu^\theta_{s'}))
    \\
    \int_0^z\int_0^s a(s,s') ds' ds & = \int_0^z \int_s^z a(s',s) ds' ds = \int_0^z \Big(\int_s^\infty - \int_z^\infty\Big)a(s',s) ds' ds 
\end{align*}
to obtain that 
\begin{equation*}
     \mathbb{E}\psi^\theta(z,v) =\mathbb{E}\psi^\theta(0,v)+\int\limits_{0}^z\mathcal{L}^\theta(s)\mathbb{E}\psi^\theta(s)\mathrm{d}s +\mathscr{R}^\theta_{p,q}(z,v)\,,
\end{equation*}
where the leading operator $\mathcal{L}^\theta(s)$ and the residual $\mathscr{R}^\theta_{p,q}(z,v)$ are given by
\begin{eqnarray}\label{eqn:L_limit_def}
    \mathcal{L}^\theta(s) & =& -\dfrac{1}{\eta\eps\theta}\dint\limits_{s}^\infty\mathbb{E}[L^\theta(s,\nu^\theta_s)L^\theta(s,\nu^\theta_{s'})]\mathrm{d}s' \\
    \mathscr{R}^\theta_{p,q}(z,v) & = &\mathscr{R}^\theta_{1}(z,v)+\mathscr{R}^\theta_{2}(z,v)+\mathscr{R}^\theta_3(z,v),
    \label{eqn:R_theta_def}
\end{eqnarray}
where
\begin{equation}\label{eqn:R123}
    \begin{aligned}
                \mathscr{R}^\theta_{1}(z,v)&=-\frac{1}{\eta\eps\theta}\int\limits_{0}^z\int\limits_{0}^s\Big(\mathbb{E}[L^\theta(s,\nu^\theta_s)L^\theta(s',\nu^\theta_{s'})\psi^\theta(s')]-\mathbb{E}[L^\theta(s,\nu^\theta_s)L^\theta(s',\nu^\theta_{s'})]\mathbb{E}\psi^\theta(s')\Big)\mathrm{d}s'\mathrm{d}s\\
                        \mathscr{R}^\theta_{2}(z,v)&=\frac{1}{\eta\eps\theta}\int\limits_{0}^z\int\limits_{0}^s\big(\mathbb{E}[L^\theta(s',\nu^\theta_{s'})L^\theta(s',\nu^\theta_s)]-\mathbb{E}[L^\theta(s,\nu^\theta_s)L^\theta(s',\nu^\theta_{s'})]\big)\mathbb{E}\psi^\theta(s')\mathrm{d}s'\mathrm{d}s\\
             \mathscr{R}^\theta_3(z,v)&=\frac{1}{\eta\eps\theta}\int\limits_{z}^\infty\int\limits_{0}^z\mathbb{E}[L^\theta(s',\nu^\theta_{s'})L^\theta(s',\nu^\theta_s)]\mathbb{E}\psi^\theta(s')\mathrm{d}s'\mathrm{d}s\,.
    \end{aligned}
\end{equation}

From~\eqref{eqn:L_theta} we observe that
\begin{equation*}
    \begin{aligned}
            &\mathbb{E}[L^\theta(z_1,\nu^\theta_{s_1})L^\theta(z_2,\nu^\theta_{s_2})]\psi\\
                 &=-\int\limits_{\mathbb{R}^{2d}}{\color{black}\frac{\mathbb{E}[\hat{\nu}^\theta_{s_1}(\mathrm{d}k)\hat{\nu}^\theta_{s_2}(\mathrm{d}k')]}{(2\pi)^{2d}}}\Big[\sum\limits_{j=1}^p\sum\limits_{l=1}^q(e^{\frac{i\eta}{\eps}[z_1g(\xi_j,k)-z_2g(\zeta_l,-k')]}\psi(\xi_j-k,\zeta_l+k')
                 \\& \qquad \qquad
                 +e^{-\frac{i\eta}{\eps}[z_1g(\zeta_l,-k)-z_2g(\xi_j,k')]}\psi(\xi_j-k',\zeta_l+k))\\
            &-\sum\limits_{j\neq j'}e^{\frac{i\eta}{\eps}[z_1g(\xi_j,k)+z_2g(\xi_{j'},k')]}\psi(\xi_j-k,\xi_{j'}-k')-\sum\limits_{l\neq l'}e^{-\frac{i\eta}{\eps}[z_1g(\zeta_l,-k)+z_2g(\zeta_{l'},-k')]}\psi(\zeta_l+k,\zeta_{l'}+k')\\
            &-\sum\limits_{j=1}^pe^{\frac{i\eta}{\eps}[z_1g(\xi_j,k)+z_2g(\xi_j-k,k')]}\psi(\xi_j-k-k')-\sum\limits_{l=1}^qe^{-\frac{i\eta}{\eps}[z_1g(\zeta_l,-k)+z_2g(\zeta_{l}+k,-k')]}\psi(\zeta_l+k+k')\Big]\,.
    \end{aligned}
\end{equation*}
Using the stationarity property of $\nu$:  $\color{black}{\mathbb{E}[\hat{\nu}^\theta_{s_1}(\mathrm{d}k)\hat{\nu}^\theta_{s_2}(\mathrm{d}k')]=(2\pi)^d\hat{\mathfrak{C}}^\theta(s_1-s_2,k)\delta(k+k')\mathrm{d}k}$, where $\mathfrak{C}^\theta$ is defined in \eqref{eqn:C_def}, we deduce from the above expression that 
\begin{equation}\label{eqn:LL_theta_expnsn}
        \begin{aligned}
            &\mathbb{E}[L^\theta(z_1,\nu^\theta_{s_1})L^\theta(z_2,\nu^\theta_{s_2})]\psi\\
            &=-\int\limits_{\mathbb{R}^d}\hat{\mathfrak{C}}^\theta(s_1-s_2,k)\Big[\sum\limits_{j=1}^p\sum\limits_{l=1}^q(e^{\frac{i\eta}{\eps}[z_1g(\xi_j,k)-z_2g(\zeta_l,k)]}+e^{-\frac{i\eta}{\eps}[z_1g(\zeta_l,k)-z_2g(\xi_j,k)]})\psi(\xi_j-k,\zeta_l-k)\\
            &-\sum\limits_{j\neq j'}e^{\frac{i\eta}{\eps}[z_1g(\xi_j,k)+z_2g(\xi_{j'},-k)]}\psi(\xi_j-k,\xi_{j'}+k)-\sum\limits_{l\neq l'}e^{-\frac{i\eta}{\eps}[z_1g(\zeta_l,k)+z_2g(\zeta_{l'},-k)]}\psi(\zeta_l-k,\zeta_{l'}+k)\\
            &-\sum\limits_{j=1}^pe^{\frac{i\eta}{\eps}[z_1g(\xi_j,k)+z_2g(\xi_j-k,-k)]}\psi-\sum\limits_{l=1}^qe^{-\frac{i\eta}{\eps}[z_1g(\zeta_l,k)+z_2g(\zeta_{l}-k,-k)]}\psi\Big]\frac{\mathrm{d}k}{(2\pi)^d}.
        \end{aligned}
\end{equation}
From~\eqref{eqn:L_limit_def} and the above calculation, we thus have that 
    \begin{equation*}
            \begin{aligned}
            \mathcal{L}^\theta(z)\psi&=\frac{1}{\eta\eps\theta}\int\limits_{\mathbb{R}^d}\int\limits_{z}^\infty\hat{\mathfrak{C}}^\theta(z-s,k)\Big[2\sum\limits_{j=1}^p\sum\limits_{l=1}^qe^{\frac{i\eta z}{\eps}[g(\xi_j,k)-g(\zeta_l,k)]}\psi(\xi_j-k,\zeta_l-k)\\
            &-\sum\limits_{j\neq j'}e^{\frac{i\eta z}{\eps}[g(\xi_j,k)+g(\xi_{j'},-k)]}\psi(\xi_j-k,\xi_{j'}+k)-\sum\limits_{l\neq l'}e^{-\frac{i\eta z}{\eps}[g(\zeta_l,k)+g(\zeta_{l'},-k)]}\psi(\zeta_l-k,\zeta_{l'}+k)\\
            &-\sum\limits_{j=1}^pe^{\frac{i\eta z}{\eps}[g(\xi_j,k)+g(\xi_j-k,-k)]}\psi-\sum\limits_{l=1}^qe^{-\frac{i\eta z}{\eps}[g(\zeta_l,k)+g(\zeta_{l}-k,-k)]}\psi\Big]\frac{\mathrm{d}s\mathrm{d}k}{(2\pi)^d}\,.
            \end{aligned}
    \end{equation*}
    This yields \eqref{eqn:L_limit_expnsn} after a change of variables $\frac{\eta(z-s)}{\eps\theta}\to s$ and noting that
    \begin{equation*}
           g(\xi,k)+g(\xi-k,-k)=|\xi|^2-|\xi-k|^2+|\xi-k|^2-|\xi-k+k|^2=0\,.
    \end{equation*}
The bound \eqref{eqn:R_theta_bound} and the proof of the theorem require lengthy estimates obtained in Lemmas \ref{lemma:R1_theta_bound}, \ref{lemma:R2_theta_bound} and \ref{lemma:R3_theta_bound} below.
\end{proof}

We note that $\mathcal{L}^\theta(z)$ is precisely the integral operator of the phase compensated moments described in~\cite{bal2024complex} under the It\^o-Schr\"{o}dinger white noise regime. The second term in the remainder comes from a mismatch in phases evaluated at different times. The third term contains correlations between points in $z$ located far apart and is small by our assumptions of decorrelations in $z$. Finding bounds for the first term is non-trivial and we will use the Duhamel representation of the solution here. We will show that this term involves pairings not consistent with the time ordering in the simplex $0\le s'_{1}\le \cdots\le s'_{2N}\le s'\le s\le z$ (and similar to the non ladder diagrams in~\cite{bal2011asymptotics}).  

Before proving \eqref{eqn:R_theta_bound}, we conclude the proofs of Theorems \ref{thm:complex_Gauss_kinetic} and \ref{thm:complex_Gauss_diff}. 
We first note that the above result translates in the physical domain for $\mu^\theta_{p,q}$ defined in \eqref{eq:mupq} as follows.
\begin{remark}
    The $p+q$th moment of $u^\theta$ can be decomposed as
    \begin{equation*}
        \mu_{p,q}^\theta(z,X,Y) =\tilde{\mathcal{B}}^\theta_{p,q}(z,X,Y)+\tilde{\mathscr{R}}_{p,q}^\theta(z,X,Y)\,,
    \end{equation*}
    with $\tilde{\mathcal{B}}^\theta_{p,q}$ given by
    \begin{equation*}
        \tilde{\mathcal{B}}^\theta_{p,q}(z,X,Y)=\int\limits_{\mathbb{R}^{(p+q)d}}\Pi^\theta(z,v)\mathcal{B}^\theta_{p,q}(z,v)e^{i(\sum_{j=1}^p\xi_j\cdot x_j-\sum_{l=1}^q\zeta_l\cdot y_l)}\prod\limits_{j=1}^p\frac{\mathrm{d}\xi_j}{(2\pi)^{d}}\prod\limits_{l=1}^q\frac{\mathrm{d}\zeta_l}{(2\pi)^{d}}\,,
    \end{equation*}
    where $\mathcal{B}^\theta_{p,q}$ is given by~\eqref{eqn:B_theta} and $\tilde{\mathscr{R}}_{p,q}^\theta$ satisfies
    \begin{equation*}
        \| \tilde{\mathscr{R}}_{p,q}^\theta(z,X,Y)\|_\infty\le c(p,q,z)\frac{\theta}{\eta^2}e^{\frac{\lambda[{\color{black}C},p,q,z,0]}{\eta^2}}\|\prod\limits_{m=1}^{p+q}\langle v_m\rangle^2\psi^\theta(0,v)\|\,.
    \end{equation*}    
\end{remark}

\paragraph{Convergence of moments: From paraxial to complex Gaussian}

Let $\gamma=(j,l)$ denote a pairing with $1\le j\le p, 1\le l\le q$. We denote by $\Lambda_\kappa$, the set of all pairings $(j,l)$, $1\le j\le p, 1\le l\le q$ such that if $\gamma,\gamma'\in\Lambda_\kappa$, then $\gamma_{1,2}\neq \gamma'_{1,2}$. For a given arrangement $\kappa$, let $m(\kappa)$ denote the number of such pairings in the set $\Lambda_\kappa$ and $K=\sum m(\kappa)=\sum\limits_{m=1}^{p\wedge q}\binom{p}{m}\binom{q}{m}m!$ denote the total number of all such distinct pairings.

Define $U^\theta_\gamma(z)$ to be a bounded operator on $\mathcal{M}_B(\mathbb{R}^{(p+q)d})$ such that $\rho(z,v)=[U^\theta_\gamma(z)\rho(0)](v) , \rho(0,v)\in\mathcal{M}_B(\mathbb{R}^{(p+q)d})$ satisfies
\begin{equation*}
    \partial_z\rho=\frac{1}{\eta^2}\int\limits_{\mathbb{R}^d}\hat{R}(k)e^{\frac{i\eta z}{\eps}[g(\xi_j,k)-g(\zeta_l,k)]}\rho(z,\xi_j-k,\zeta_l-k)\frac{\mathrm{d}k}{(2\pi)^d}\,.
\end{equation*}
Also, let $U_\eta(z)=e^{-\frac{R(0)z}{\eta^2}}$. 
\begin{theorem}\label{thm:Psi_Gaussian_decomp}
The $p+q$th moments of the phase compensated field satisfy the decomposition
\begin{equation}\label{eqn:Psi_Gaussian}
    \Psi^\theta_{p,q}(z,v)=[N_{p,q}^\theta(z)\psi^\theta(0)](v)+E^\theta_{p,q}(z,v)\,,
\end{equation}    
where $N_{p,q}^\theta(z)$ is a bounded operator on $\mathcal{M}_B(\mathbb{R}^{(p+q)d})$ given by
\begin{equation}\label{eqn:N}
    N_{p,q}^\theta(z)=U_\eta^{\frac{p+q}{2}}(z)\sum\limits_{\kappa=1}^{K}\prod\limits_{\gamma\in\Lambda_\kappa}(U^\theta_\gamma(z)-\mathbb{I})+U_\eta^{\frac{p+q}{2}}(z)\,,
\end{equation}
and $E^\theta_{p,q}$ satisfies
\begin{equation}\label{eqn:E_bound}
    \|E_{p,q}^\theta(z,v)\|\le c(p,q,z)(\eps^{\frac 13}+\theta^{\frac 12})\|\prod\limits_{m=1}^{p+q}\langle v_m\rangle^2\psi^\theta(0)\|\,.
\end{equation}
\begin{proof}
From Theorem~\ref{thm:psi_theta_decomp}, we have that 
    \begin{equation*}
        \Psi_{p,q}^\theta(z,v) =\mathcal{B}^\theta_{p,q}(z,v)+\mathscr{R}^\theta_{p,q}(z,v)\,,
    \end{equation*}
    with $\mathcal{B}^\theta_{p,q}$ as in~\eqref{eqn:B_theta} and $\mathscr{R}^\theta$ satisfying the bound~\eqref{eqn:R_theta_bound}. Note that $\mathcal{B}_{p,q}^\theta$ follows the same evolution equation as that of the phase compensated moments of the It\^o-Schr\"{o}dinger equation.  Due to this, the results in~\cite{bal2024complex} apply and using \cite[Theorem 4.2]{bal2024complex}, we have that
    \begin{equation*}
        \mathcal{B}^\theta_{p,q}(z,v)=[N^\theta_{p,q}(z)\psi^\theta(0)](v)+\mathcal{E}_{p,q}^\theta(z,v)\,,
    \end{equation*}
    with $N_{p,q}^\theta(z)$ as defined in~\eqref{eqn:N} and $\mathcal{E}_{p,q}^\theta$ satisfying
    \begin{equation*}
        \|\mathcal{E}_{p,q}^\theta(z,v)\|\le c(p,q,z)\eps^{\frac 13}\|\prod\limits_{m=1}^{p+q}\psi^\theta(0)\|\,.
    \end{equation*}
    This gives
    \begin{equation*}
        E^\theta_{p,q}(z,v)=\mathscr{R}_{p,q}^\theta(z,v)+\mathcal{E}_{p,q}^\theta(z,v)
    \end{equation*}
 with $E^\theta_{p,q}$ satisfying the bound in~\eqref{eqn:E_bound}.
\end{proof}
\end{theorem}
Define the inverse (phase re-compensated) Fourier transform
\begin{equation}\label{eqn:N_tilde}
    \tilde{N}^\theta_{p,q}(z,X,Y)=\int\limits_{\mathbb{R}^{(p+q)d}}\Pi^\theta(z,v)[N_{p,q}^\theta(z)\psi^\theta(0)](v)e^{i(\sum_{j=1}^p\xi_j\cdot x_j-\sum_{l=1}^q\zeta_l\cdot y_l)}\prod\limits_{j=1}^p\frac{\mathrm{d}\xi_j}{(2\pi)^{d}}\prod\limits_{l=1}^q\frac{\mathrm{d}\zeta_l}{(2\pi)^{d}}\,,
\end{equation}
with $\tilde{E}^\theta_{p,q}(z,X,Y)$ defined in a similar manner for $E^\theta_{p,q}$.
Then we have the following direct consequence of the above results in physical domain:
\begin{theorem}\label{thm:mu_theta_decomp_final}
  The $p+q$th moments of $u^\theta$ can be written as
  \begin{equation*}
      \mu^\theta_{p,q}(z,X,Y)=\tilde{N}_{p,q}^\theta(z,X,Y)+\tilde{E}^\theta_{p,q}(z,X,Y)\,,
  \end{equation*}
  with $\tilde{N}_{p,q}^\theta$ as defined in~\eqref{eqn:N_tilde} and
  \begin{equation*}
      \|\tilde{E}_{p,q}^\theta(z,X,Y)\|_\infty\le c(p,q,z)(\eps^{\frac 13}+\theta^{\frac 12})\|\prod\limits_{m=1}^{p+q}\langle v_m\rangle^2\psi^\theta(0)\|\,.
  \end{equation*}
\end{theorem}

\begin{proof}[Proof of Theorems \ref{thm:complex_Gauss_kinetic} and \ref{thm:complex_Gauss_diff}]
    These are direct consequences of the  estimate in Theorem \ref{thm:mu_theta_decomp_final} combined with the corresponding results in the It\^o-Schr\"odinger regime \cite[Theorems 2.2 \& 2.4]{bal2024complex}.
\end{proof}

It thus remains to conclude the proof of Theorem \ref{thm:psi_theta_decomp} by  proving the error bound \eqref{eqn:R_theta_bound}. We first make the following preparations. 
Recall that $\|\cdot\|$ denotes the total variation norm. 

\paragraph{Control of the error terms in \eqref{eqn:R_theta_bound}} 

We now state and prove the lemmas allowing us to conclude the proof of Theorem \ref{thm:psi_theta_decomp}.

\begin{lemma}\label{lemma:R1_theta_bound}
We have the first estimate:
    \begin{equation}\label{eqn:R1_theta_bound}
        \|\prod\limits_{m=1}^{p+q}\langle v_m\rangle^\alpha \mathscr{R}^\theta_{1}(z,v)\|\le c_1(p,q,z)\frac{\eps\theta}{\eta^3}\|\prod\limits_{m=1}^{p+q}\langle v_m\rangle^\alpha\psi^\theta(0,v)\|e^{\frac{\lambda[{\color{black}C},p,q,z,\alpha]}{\eta^2}}\,,
    \end{equation}
 where $\lambda$ is defined in~\eqref{eqn:lambda_def}, $0\le \alpha\le \frac{d_0+1}{p+q}$ and $c_1$ is a constant independent of $\theta$.
 \end{lemma} 
 \begin{proof}
Plugging the chaos expansion \eqref{eqn:chaos_series}-\eqref{eqn:chaos_series2} in \eqref{eqn:R123} and using the Gaussian property {\color{black}
\begin{equation*}
    \begin{aligned}
        &\mathbb{E}[\hat{\nu}^\theta_s(\mathrm{d}k)\hat{\nu}^\theta_{s'}(\mathrm{d}k')\prod\limits_{j=1}^{2N}\hat{\nu}^\theta_{s_j}(\mathrm{d}k_j)]- \mathbb{E}[\hat{\nu}^\theta_s(\mathrm{d}k)\hat{\nu}^\theta_{s'}(\mathrm{d}k')]\mathbb{E}[\prod\limits_{j=1}^{2N}\hat{\nu}^\theta_{s_j}(\mathrm{d}k_j)]\\
        &=\sum\limits_{\substack{m,n=1\\ n\neq m}}^{2N}\mathbb{E}[\hat{\nu}^\theta_s(\mathrm{d}k)\hat{\nu}^\theta_{s_m}(\mathrm{d}k_m)]\mathbb{E}[\hat{\nu}^\theta_{s'}(\mathrm{d}k')\hat{\nu}^\theta_{s_n}(\mathrm{d}k_n)]\mathbb{E}[\prod_{\substack{j=1\\ j\neq m,n}}^{2N}\hat{\nu}^\theta_{s_j}(\mathrm{d}k_j)]\,,
    \end{aligned}
\end{equation*}}
we deduce that
{\color{black}
 \begin{equation*}
     \begin{aligned}
          \mathscr{R}^\theta_{1}(z,v)&=\sum\limits_{N\ge 1}\frac{1}{(\eta\eps\theta)^{N+1}}\int\limits_{0}^z\int\limits_{0}^s\int\limits_{[0,s']^{2N}_<}\int\limits_{\mathbb{R}^{(2N+2)d}}\sum_{\substack{m,n=1\\ n\neq m}}^{2N}{\color{black}\mathbb{E}[\hat{\nu}^\theta_s(\mathrm{d}k)\hat{\nu}^\theta_{s_m}(\mathrm{d}k_m)]\mathbb{E}[\hat{\nu}^\theta_{s'}(\mathrm{d}k')\hat{\nu}^\theta_{s_n}(\mathrm{d}k_n)]}\\
        &\times{\color{black}\mathbb{E}[\prod_{\substack{j=1\\ j\neq m,n}}^{2N}\hat{\nu}^\theta_{s_j}(\mathrm{d}k_j)]}\sum\limits_{j=1}^{(p+q)^{2N+2}}e^{\frac{i\eta }{\eps}G_j(\Vec{s},\Vec{k},v)}\psi^\theta(0,v-A_j\Vec{k})\frac{\mathrm{d}s'\mathrm{d}s}{(2\pi)^{2d}}\prod\limits_{j=1}^{2N}\frac{\mathrm{d}s_j}{(2\pi)^d}\,.
     \end{aligned}
 \end{equation*}}
Again using the Gaussian structure of the potential, we deduce using \eqref{eqn:chaos_2N} that
        \begin{equation*}
            \begin{aligned}
               & \|\prod\limits_{m=1}^{p+q}\langle v_m\rangle^\alpha \mathscr{R}^\theta_{1}(z,v)\|
                 \le   \sum\limits_{N\ge 1}\frac{1}{(\eta\eps\theta)^{N+1}}\int\limits_{0}^z\int\limits_{0}^s\int\limits_{[0,s']^{2N}_<}\int\limits_{\mathbb{R}^{(N+1)d}}\sum_{\substack{m,n=1\\ n\neq m}}^{2N}\sum\limits_{P^{2N-2}}\sum\limits_{j=1}^{(p+q)^{2N+2}}\|\prod\limits_{m'=1}^{p+q}\langle v_{m'}\rangle^\alpha \psi^\theta(0,v-A_j'\Vec{k})\| \\
       & \times{\color{black}|\hat{\mathfrak{C}}^\theta(s-s_m,k)||\hat{\mathfrak{C}}^\theta(s'-s_n,k')|}\frac{\mathrm{d}k\mathrm{d}k'}{(2\pi)^{2d}}\prod\limits_{\substack{(j',l')\in P^{2N-2}\\j',l'\neq m,n}}{\color{black}|\hat{\mathfrak{C}}^\theta(s_{j'}-s_{l'},k_{j'})|}\frac{\mathrm{d}k_{j'}}{(2\pi)^d}\prod\limits_{j=1}^{2N}\mathrm{d}s_j\mathrm{d}s'\mathrm{d}s\,.
            \end{aligned}
        \end{equation*}
        Note that as in~\eqref{eqn:v_m_alpha_prod1} 
        \begin{equation*}
            \begin{aligned}
                \|\prod\limits_{m=1}^{p+q}\langle v_m\rangle^\alpha \psi^\theta(0,v-A_j'\Vec{k})\|&=\|\prod\limits_{m=1}^{p+q}\langle v_m+\sum\limits_{l=1}^{N+1}[A_j']_{m,l}k_m\rangle^\alpha \psi^\theta(0,v)\|\,,
            \end{aligned}
        \end{equation*}
        with $k_{N},k_{N+1}$ taken as $k',k$ respectively. From Lemma~\ref{lemma:chaos}, $|[A_j']_{m,l}|\le C$ and we have as in~\eqref{eqn:v_m_alpha_prod2} that
        \begin{equation*}
          1\le  \langle v_m+\sum\limits_{l=1}^{N+1}[A_j']_{m,l}k_m\rangle\le C^{N+1}\langle v_m\rangle\prod\limits_{l=1}^{N+1}\langle k_l\rangle.
        \end{equation*}
        Using this, we have that $\|\prod\limits_{m=1}^{p+q}\langle v_m\rangle^\alpha \mathscr{R}^\theta_{1}(z,v)\|$ is bounded by
         \begin{equation*}
            \begin{aligned}
            &\sum\limits_{N\ge 1}\frac{C^{(N+1)(p+q)}(p+q)^{2N+2}}{(\eta\eps\theta)^{N+1}}\|\prod\limits_{m=1}^{p+q}\langle v_m\rangle^\alpha \psi^\theta(0,v)\|\\
            &\times\int\limits_{0}^z\int\limits_{0}^s\int\limits_{\mathbb{R}^{2d}}\sum_{\substack{m,n=1\\ n\neq m}}^{2N} ( \langle k\rangle \langle k'\rangle)^{(p+q)\alpha}{\color{black}|\hat{\mathfrak{C}}^\theta(s-s_m,k)||\hat{\mathfrak{C}}^\theta(s'-s_n,k')|}\frac{\mathrm{d}k\mathrm{d}k'}{(2\pi)^{2d}}\mathrm{d}s'\mathrm{d}s\\
        &\times\int\limits_{[0,s']^{2N}_<}\int\limits_{\mathbb{R}^{(N-1)d}}\sum\limits_{P^{2N-2}}\prod\limits_{\substack{(j',l')\in P^{2N-2}\\j',l'\neq m,n}}\langle k_{j'}\rangle^{(p+q)\alpha}{\color{black}|\hat{\mathfrak{C}}^\theta(s_{j'}-s_{l'},k_{j'})|}\frac{\mathrm{d}k_{j'}}{(2\pi)^d}\prod\limits_{j=1}^{2N}\mathrm{d}s_j\,.
            \end{aligned}
        \end{equation*}
        Let 
        $\mathfrak{C}_{n,\alpha}(s) =\int\limits_{\mathbb{R}^d}\langle k\rangle^{n\alpha}{\color{black}|\hat{\mathfrak{C}}(s,k)|}\frac{\mathrm{d}k}{(2\pi)^d}$
        with $\mathfrak{C}^\theta_{n,\alpha}(s)=\mathfrak{C}_{n,\alpha}\big(\frac{\eta s}{\eps\theta}\big)$.  Then we have that
                 \begin{equation*}
            \begin{aligned}
                \|\prod\limits_{m=1}^{p+q}\langle v_m\rangle^\alpha \mathscr{R}^\theta_{1}(z,v)\| \ \le\  \|\prod\limits_{m=1}^{p+q}\langle v_m\rangle^\alpha \psi^\theta(0,v)\|\sum\limits_{N\ge 1}\frac{C^{(N+1)(p+q)}(p+q)^{2N+2}}{(\eta\eps\theta)^{N+1}}
                \\ \times \int\limits_{0}^z\int\limits_{0}^s\int\limits_{[0,s']^{2N}_<}\sum_{\substack{m,n=1\\ n\neq m}}^{2N}\mathfrak{C}_{p+q,\alpha}^\theta(s-s_m)\mathfrak{C}_{p+q,\alpha}^\theta(s'-s_n)
                \sum\limits_{P^{2N-2}}\prod\limits_{\substack{(j,l)\in P^{2N-2}\\j,l\neq m,n}}\mathfrak{C}_{p+q,\alpha}^\theta(s_j-s_l)\prod\limits_{j=1}^{2N}\mathrm{d}s_j\mathrm{d}s'\mathrm{d}s\,.
            \end{aligned}
        \end{equation*}
        Again, symmetrizing the simplex leads to an expression of the form
                         \begin{equation}
            \begin{aligned}
                &\|\prod\limits_{m=1}^{p+q}\langle v_m\rangle^\alpha \mathscr{R}^\theta_{1}(z,v)\|\le \|\prod\limits_{m=1}^{p+q}\langle v_m\rangle^\alpha \psi^\theta(0,v)\|\sum\limits_{N\ge 1}\frac{C^{(N+1)(p+q)}(p+q)^{2N+2}}{(N-1)!}\\
                &\times \Big(\frac{1}{\eta\eps\theta}\int\limits_{[0,z]^2}\mathfrak{C}^\theta_{p+q,\alpha}(s-s')\mathrm{d}s'\mathrm{d}s\Big)^{N-1}
                \Big(\frac{1}{(\eta\eps\theta)^2}\int\limits_{0}^z\int\limits_{0}^s\int\limits_{[0,s']^{2}}\mathfrak{C}_{p+q,\alpha}^\theta(s-s_1)\mathfrak{C}_{p+q,\alpha}^\theta(s'-s_2)\mathrm{d}s_2\mathrm{d}s_1\mathrm{d}s'\mathrm{d}s\Big).
            \end{aligned}
        \end{equation}
        As before, we have that
        \begin{equation*}
            \frac{1}{\eta\eps\theta}\int\limits_{[0,z]^2}\mathfrak{C}^\theta_{p+q,\alpha}(s-s')\mathrm{d}s'\mathrm{d}s\le \frac{z}{\eta^2}\int\limits_{\mathbb{R}}\mathfrak{C}_{p+q,\alpha}(s)\mathrm{d}s=\frac{z}{\eta^2}{\color{black}\int_{\mathbb{R}^{d+1}}\langle k\rangle^{(p+q)\alpha}|\hat{C}(s,k)|\mathrm{d}s\mathrm{d}k}.
        \end{equation*}
        On the other hand,
        \begin{equation*}
            \begin{aligned}
               &\frac{1}{(\eta\eps\theta)^2}\int\limits_{0}^z\int\limits_{0}^s\int\limits_{[0,s']^{2}}\mathfrak{C}_{p+q,\alpha}^\theta(s-s_1)\mathfrak{C}_{p+q,\alpha}^\theta(s{\color{black}'}-s_2)\mathrm{d}s_2\mathrm{d}s_1\mathrm{d}s'\mathrm{d}s \\
               &\le \Big(\frac{1}{\eta\eps\theta}\int\limits_{0}^z\int\limits_{0}^s\int\limits_{0}^{s'}\mathfrak{C}_{p+q,\alpha}^\theta(s-s_1)\mathrm{d}s_1\mathrm{d}s'\mathrm{d}s\Big)\Big(\frac{1}{\eta^2}\int\limits_{\mathbb{R}}\mathfrak{C}_{p+q,\alpha}(s)\mathrm{d}s\Big)\\
                 &={\color{black}\Big(\frac{1}{\eta\eps\theta}\int\limits_{0}^z\int\limits_{0}^s\int\limits_{s_1}^s\mathfrak{C}^\theta_{p+q,\alpha}(s-s_1)\mathrm{d}s'\mathrm{d}s_1\mathrm{d}s\Big)\Big(\frac{1}{\eta^2}\int_{\mathbb{R}^{d+1}}\langle k\rangle^{(p+q)\alpha}|\hat{C}(s,k)|\mathrm{d}s\mathrm{d}k\Big)}.
            \end{aligned}
        \end{equation*}
So, we deduce that the above integral is bounded as
        \begin{equation*}
            \begin{aligned}
               & {\color{black}\Big(\frac{1}{\eta\eps\theta}\int\limits_{0}^z\int\limits_{0}^s(s-s_1)\mathfrak{C}^\theta_{p+q,\alpha}(s_1)\mathrm{d}s_1\mathrm{d}s\Big)\Big(\frac{1}{\eta^2}\int_{\mathbb{R}^{d+1}}\langle k\rangle^{(p+q)\alpha}|\hat{C}(s,k)|\mathrm{d}s\mathrm{d}k\Big)}\\
                &{\color{black}\le\Big(\frac{\eps\theta z}{\eta^3}\int\limits_{\mathbb{R}^{d+1}}|s|\langle k\rangle^{(p+q)\alpha} |\hat{C}(s,k)|\mathrm{d}s\mathrm{d}k\Big)\Big(\frac{1}{\eta^2}\int_{\mathbb{R}^{d+1}}\langle k\rangle^{(p+q)\alpha}|\hat{C}(s,k)|\mathrm{d}s\mathrm{d}k\Big)}\,.
            \end{aligned}
        \end{equation*}
        Together, these lead to a bound of the form
        \begin{equation*}
            \begin{aligned}
                &\|\prod\limits_{m=1}^{p+q}\langle v_m\rangle^\alpha \mathscr{R}^\theta_{1}(z,v)\|\\
                &\le \frac{\eps\theta}{\eta^3}\|\prod\limits_{m=1}^{p+q}\langle v_m\rangle^\alpha \psi^\theta(0,v)\|\sum\limits_{N\ge 1}\frac{(p+q)^{2N+2}}{(N-1)!}C^{(N+1)(p+q)}\frac{z^N}{\eta^{2N}} \Big({\color{black}\int\limits_{\mathbb{R}^{d+1}}\langle s\rangle\langle k\rangle^{(p+q)\alpha} |\hat{C}(s,k)|\mathrm{d}s\mathrm{d}k}\Big)^N\\
                &\le c_1(p,q,z)\frac{\eps\theta }{\eta^3}\|\prod\limits_{m=1}^{p+q}\langle v_m\rangle^\alpha \psi^\theta(0,v)\|e^{\frac{\lambda[{\color{black}C},p,q,z,\alpha]}{\eta^2}} \,.
            \end{aligned}
        \end{equation*}
               This leads to~\eqref{eqn:R1_theta_bound} and concludes the derivation of the first estimate.
\end{proof}

\begin{lemma}\label{lemma:R2_theta_bound} We have the following second estimate:
    \begin{equation}\label{eqn:R2_theta_bound}
        \|\prod\limits_{m=1}^{p+q}\langle v_m\rangle^\alpha \mathscr{R}^\theta_{2}(z,v)\|\le c_2(p,q,z)\frac{\theta}{\eta^2}e^{\frac{\lambda[{\color{black}C},p,q,z,\alpha]}{\eta^2}} \|\prod\limits_{m=1}^{p+q}\langle v_m\rangle^{1+\alpha}\psi^\theta(0)\|\,,
    \end{equation}
      with $\lambda$ as in~\eqref{eqn:lambda_def} and $0\le\alpha\le \frac{d_0+1}{p+q}$. Here, $c_2$ is a constant independent of $\theta$. 
\end{lemma}
\begin{proof}
        From~\eqref{eqn:R123} and~\eqref{eqn:LL_theta_expnsn}, we have
        \begin{equation*}
            \begin{aligned}
                &\mathscr{R}^\theta_{2}(z,v)=\frac{1}{\eta\eps\theta}\int\limits_{0}^z\int\limits_{0}^s\int\limits_{\mathbb{R}^d}\hat{\mathfrak{C}}^\theta(s-s',k)\\
                &\times\Big[\sum\limits_{j,l}\big(2e^{\frac{i\eta}{\eps}s'[g(\xi_j,k)-g(\zeta_l,k)]}-e^{\frac{i\eta }{\eps}[sg(\xi_j,k)-s'g(\zeta_l,k)]}-e^{-\frac{i\eta}{\eps}[sg(\zeta_l,k)-s'g(\xi_j,k)]}\big)\mathbb{E}\psi^\theta(s',\xi_j-k,\zeta_l-k)\\
                &-\sum\limits_{j\neq j'}\big(e^{\frac{i\eta s'}{\eps}[g(\xi_j,k)+g(\xi_{j'},-k)]}-e^{\frac{i\eta }{\eps}[sg(\xi_j,k)+s'g(\xi_{j'},-k)]}\big)\mathbb{E}\psi^\theta(s',\xi_j-k,\xi_{j'}+k)\\
                &-\sum\limits_{l\neq l'}\big(e^{-\frac{i\eta s'}{\eps}[g(\zeta_l,k)+g(\zeta_{l'},-k)]}-e^{-\frac{i\eta }{\eps}[sg(\zeta_l,k)+s'g(\zeta_{l'},-k)]}\big)\mathbb{E}\psi^\theta(s',\zeta_l-k,\zeta_{l'}+k)\\
                &-\sum\limits_{j}\big(1-e^{\frac{i\eta}{\eps}[sg(\xi_j,k)+s'g(\xi_j-k,-k)]}\big)\mathbb{E}\psi^\theta(s')-\sum\limits_{l}\big(1-e^{-\frac{i\eta}{\eps}[sg(\zeta_l,k)+s'g(\zeta_l-k,-k)]}\big)\mathbb{E}\psi^\theta(s')\Big]\frac{\mathrm{d}k}{(2\pi)^d}\mathrm{d}s'\mathrm{d}s\,.
            \end{aligned}
        \end{equation*}
        Switching the order of integration of $s'$ and $s$ and performing a change of variables $\frac{\eta(s-s')}{\eps\theta}\to s$ gives
        \begin{equation*}
            \begin{aligned}
                \mathscr{R}^\theta_{2}(z,v)&=\frac{1}{\eta^2}\int\limits_{0}^z\int\limits_{0}^{\frac{\eta(z-s')}{\eps\theta}}\int\limits_{\mathbb{R}^d}\hat{\mathfrak{C}}(s,k)\\
                &\times\Big[\sum\limits_{j,l}e^{\frac{i\eta}{\eps}s'[g(\xi_j,k)-g(\zeta_l,k)]}\big(2-e^{i\theta s g(\xi_j,k)}-e^{-i\theta sg(\zeta_l,k)}\big)\mathbb{E}\psi^\theta(s',\xi_j-k,\zeta_l-k)\\
                &-\sum\limits_{j\neq j'}e^{\frac{i\eta s'}{\eps}[g(\xi_j,k)+g(\xi_{j'},-k)]}\big(1-e^{i\theta sg(\xi_j,k)}\big)\mathbb{E}\psi^\theta(s',\xi_j-k,\xi_{j'}+k)\\
                &-\sum\limits_{l\neq l'}e^{-\frac{i\eta s'}{\eps}[g(\zeta_l,k)+g(\zeta_{l'},-k)]}\big(1-e^{-i\theta sg(\zeta_l,k)}\big)\mathbb{E}\psi^\theta(s',\zeta_l-k,\zeta_{l'}+k)\\
                &-\sum\limits_{j}\big(1-e^{i\theta sg(\xi_j,k)}\big)\mathbb{E}\psi^\theta(s')-\sum\limits_{l}\big(1-e^{-i\theta sg(\zeta_l,k)}\big)\mathbb{E}\psi^\theta(s')\Big]\frac{\mathrm{d}k}{(2\pi)^d}\mathrm{d}s\mathrm{d}s'\,.
            \end{aligned}
        \end{equation*}
        From the above expression, we have that
            \begin{equation*}
            \begin{aligned}
               &\|\prod\limits_{m=1}^{p+q}\langle v_m\rangle^\alpha \mathscr{R}^\theta_{2}(z,v)\|\\
               &\le\frac{\theta}{\eta^2}\int\limits_{0}^z\int\limits_{0}^{\frac{\eta(z-s')}{\eps\theta}}\int\limits_{\mathbb{R}^d}|s|{\color{black}|\hat{\mathfrak{C}}(s,k)|}\int\limits_{\mathbb{R}^{(p+q)d}}\Big[\sum\limits_{j,l}\prod\limits_{m=1}^{p+q}\langle v_m\rangle^\alpha\big(|g(\xi_j,k)|+|g(\zeta_l,k)|\big)|\mathbb{E}\psi^\theta(s',\xi_j-k,\zeta_l-k)|\\
                &+\sum\limits_{j\neq j'}\prod\limits_{m=1}^{p+q}\langle v_m\rangle^\alpha|g(\xi_j,k)||\mathbb{E}\psi^\theta(s',\xi_j-k,\xi_{j'}+k)|+\sum\limits_{l\neq l'}\prod\limits_{m=1}^{p+q}\langle v_m\rangle^\alpha|g(\zeta_l,k)||\mathbb{E}\psi^\theta(s',\zeta_l-k,\zeta_{l'}+k)|\\
                &+\sum\limits_{j}\prod\limits_{m=1}^{p+q}\langle v_m\rangle^\alpha|g(\xi_j,k)||\mathbb{E}\psi^\theta(s')|+\sum\limits_{l}\prod\limits_{m=1}^{p+q}\langle v_m\rangle^\alpha|g(\zeta_l,k)||\mathbb{E}\psi^\theta(s')|\Big]\frac{\mathrm{d}k}{(2\pi)^d}\mathrm{d}s\mathrm{d}s'\,.
            \end{aligned}
        \end{equation*}
        For a fixed $(j,l)$ in the first sum, we make a change of variables $\xi_j-k\to \xi_j, \zeta_l-k\to \zeta_l$. This transformation leads to a product of the form
        \begin{equation*}
            \Big(\prod_{\substack{m=1\\ m\neq j,l}}^{2N}\langle v_m\rangle^\alpha\Big)\langle \xi_j+k\rangle^\alpha\langle \zeta_l+k\rangle^\alpha\Big(|g(\xi_j+k,k)|+|g(\zeta_l+k,k)|\Big).
        \end{equation*}
        Since $1\le  \langle \xi_j+k\rangle\langle \zeta_l+k\rangle\le 2\langle \xi_j\rangle\langle \zeta_l\rangle\langle k\rangle^2$
        and
        \begin{equation*}
            \begin{aligned}
               |g(\xi_j+k,k)|+|g(\zeta_l+k,k)|&\le 2|k|^2+2|k|(|\xi_j|+|\zeta_l|)\le 2\langle k\rangle^2 +2\langle k\rangle(\langle \xi_j\rangle+ \langle \zeta_l\rangle)\le C\langle k\rangle^2\langle \xi_j\rangle \langle \zeta_l\rangle\,,
            \end{aligned}
        \end{equation*}
        we have that
          \begin{equation*}
            \begin{aligned}
              & \|\prod\limits_{m=1}^{p+q}\langle v_m\rangle^\alpha \mathscr{R}^\theta_{2}(z,v) \| \\
              &\le C\frac{\theta}{\eta^2}\int\limits_{0}^z\int\limits_{0}^{\frac{\eta(z-s')}{\eps\theta}}\int\limits_{\mathbb{R}^d}|s|\langle k\rangle^{2+2\alpha}{\color{black}|\hat{\mathfrak{C}}(s,k)|} \|\sum\limits_{\substack{n,n'=1\\ n\neq n'}}^{p+q}\langle v_n\rangle\langle v_{n'}\rangle\Big(\prod\limits_{m=1}^{p+q}\langle v_m\rangle^\alpha\Big)\mathbb{E}\psi^\theta(s',v)\|\frac{\mathrm{d}k}{(2\pi)^d}\mathrm{d}s\mathrm{d}s'\\
               &\le Cz\frac{\theta}{\eta^2}\sup\limits_{0\le s\le z} \|\sum\limits_{\substack{n,n'=1\\ n\neq n'}}^{p+q}\langle v_n\rangle\langle v_{n'}\rangle\Big(\prod\limits_{m=1}^{p+q}\langle v_m\rangle^\alpha\Big)\mathbb{E}\psi^\theta(s,v)\|{\color{black}\int\limits_{\mathbb{R}^{d+1}}|s|\langle k\rangle^{2+2\alpha}|\hat{C}(s,k)|\mathrm{d}s\mathrm{d}k}\,.
            \end{aligned}
        \end{equation*}
    The bound in~\eqref{eqn:R2_theta_bound} now follows from assumption~\eqref{eqn:C_n_assump} and using Lemma \ref{lem:E_psi_theta_bounds} applied to
    \[
        \|\sum\limits_{\substack{n,n'=1\\ n\neq n'}}^{p+q}\langle v_n\rangle\langle v_{n'}\rangle\Big(\prod\limits_{m=1}^{p+q}\langle v_m\rangle^\alpha\Big)\mathbb{E}\psi^\theta(s,v)\|.
    \]
 This concludes the proof of the Lemma.
\end{proof}
The final error term in \eqref{eqn:R_theta_bound} is bounded as follows:
\begin{lemma}\label{lemma:R3_theta_bound} We have the following third estimate:
    \begin{equation}\label{eqn:R3_theta_bound}
        \|\prod\limits_{m=1}^{p+q}\langle v_m\rangle^\alpha \mathscr{R}^\theta_3(z,v)\|\le c_3(p,q)\frac{\eps\theta}{\eta^3}\|\prod\limits_{m=1}^{p+q}\langle v_m\rangle^{\alpha} \psi^\theta(0,v)\|e^{\frac{\lambda[{\color{black}C},p,q,z,\alpha]}{\eta^2}} \,,
    \end{equation}
  with $\lambda$ as in~\eqref{eqn:lambda_def} and $0\le\alpha\le \frac{d_0+1}{p+q}$. 
    \begin{proof}
        From~\eqref{eqn:R123} and~\eqref{eqn:LL_theta_expnsn}, we have that
        \begin{equation*}
            \begin{aligned}
                \|\prod\limits_{m=1}^{p+q}\langle v_m\rangle^\alpha \mathscr{R}^\theta_3(z,v)\|&\le\frac{1}{\eta\eps\theta}\int\limits_{z}^\infty\int\limits_{0}^z\int\limits_{\mathbb{R}^d}{\color{black}|\hat{\mathfrak{C}}^\theta(s-s',k)|}\Big[2\sum\limits_{j,l}\|\prod\limits_{m=1}^{p+q}\langle v_m\rangle^\alpha \mathbb{E}\psi^\theta(s',\xi_j-k,\zeta_l-k)\|\\
                &+\sum\limits_{j\neq j'}\|\prod\limits_{m=1}^{p+q}\langle v_m\rangle^\alpha \mathbb{E}\psi^\theta(s',\xi_j-k,\xi_{j'}+k)\|+\sum\limits_{l\neq l'}\|\prod\limits_{m=1}^{p+q}\langle v_m\rangle^\alpha \mathbb{E}\psi^\theta(s',\zeta_l-k,\zeta_{l'}+k)\|\\
                &+(p+q)\|\prod\limits_{m=1}^{p+q}\langle v_m\rangle^\alpha \mathbb{E}\psi^\theta(s',v)\|\Big]\frac{\mathrm{d}k}{(2\pi)^d}\mathrm{d}s'\mathrm{d}s\,.
            \end{aligned}
        \end{equation*}
        As before, making an appropriate change of variables gives for instance
        \begin{equation*}
            \|\prod\limits_{m=1}^{p+q}\langle v_m\rangle^\alpha \mathbb{E}\psi^\theta(s',\xi_j-k,\zeta_l-k)\|= \|\prod\limits_{m=1, m\neq j, p+l}^{p+q}\langle v_m\rangle^\alpha\langle\xi_j+k\rangle^\alpha\langle\zeta_l+k\rangle^\alpha \mathbb{E}\psi^\theta(s',v)\|,
        \end{equation*}
        which is bounded by $C\langle k\rangle^{2\alpha}\|\prod\limits_{m=1}^{p+q}\langle v_m\rangle^\alpha\mathbb{E}\psi^\theta(s',v)\|$.
        This gives
        \begin{equation*}
            \begin{aligned}
                \|\prod\limits_{m=1}^{p+q}\langle v_m\rangle^\alpha \mathscr{R}^\theta_3(z,v)\|&\le\frac{(p+q)^2}{\eta\eps\theta}\sup\limits_{0\le s\le z}\|\prod\limits_{m=1}^{p+q}\langle v_m\rangle^\alpha\mathbb{E}\psi^\theta(s,v)\|{\color{black}\int\limits_{0}^z\int\limits_{z}^\infty\int\limits_{\mathbb{R}^d}\langle k\rangle^{2\alpha}|\hat{\mathfrak{C}}^\theta(s-s',k)|\frac{\mathrm{d}k}{(2\pi)^d}\mathrm{d}s\mathrm{d}s'}\,.
            \end{aligned}
        \end{equation*}
{\color{black}After a change of variables $s-s'\to s$, we have
\begin{equation*}
    \begin{aligned}
        \int_0^z\int_z^\infty |\hat{C}^\theta(s-s',\cdot)|\mathrm{d}s\mathrm{d}s'&=\Big(\int_0^z\int_{z-s'}^z+\int_0^z\int_z^\infty\Big)|\hat{C}^\theta(s,k)|\mathrm{d}s\mathrm{d}s'\\
        &=\int_0^zs|\hat{C}^\theta(s,\cdot)|\mathrm{d}s+z\int_z^\infty|\hat{C}^\theta(s,\cdot)|\mathrm{d}s\le \int_0^\infty s|\hat{C}^\theta(s,\cdot)|\mathrm{d}s\,.
    \end{aligned}
\end{equation*}
From a change of variables $\frac{\eta s}{\eps\theta}\to s$, we have
              \begin{equation*}
            \begin{aligned}
                \|\prod\limits_{m=1}^{p+q}\langle v_m\rangle^\alpha \mathscr{R}^\theta_3(z,v)\|&\le C(p+q)^2\frac{\eps\theta}{\eta^3}\sup\limits_{0\le s\le z}\|\prod\limits_{m=1}^{p+q}\langle v_m\rangle^\alpha\mathbb{E}\psi^\theta(s,v)\|\int\limits_{\mathbb{R}^{d+1}}|s|\langle k\rangle^{2\alpha}|\hat{C}(s,k)|\mathrm{d}s\mathrm{d}k\,.
            \end{aligned}
        \end{equation*}
        }
              The bound in~\eqref{eqn:R3_theta_bound} now follows from using Lemma \ref{lem:E_psi_theta_bounds} in the above expression. 
        This concludes the derivation of the third estimate and hence the proof of Theorem \ref{thm:psi_theta_decomp}.
    \end{proof}
\end{lemma}

\subsection{Proof of the tightness result}\label{subsec:tightness_proof}
\begin{proof}[Proof of Theorem~\ref{thm:tightness}]
        We follow a similar strategy to the proof of \cite[Theorem 2.7]{bal2024complex}. Let $\delta_h\phi^\theta(z,r,x):=\phi^\theta(z,r,x+h)-\phi^\theta(z,r,x)$.
        Then, 
        \begin{equation*}
        \begin{aligned}
                       & \mathbb{E}|\phi^\theta(z,r,x+h)-\phi^\theta(z,r,x)|^{2n}=\mathbb{E}|\delta_h\phi^\theta(z,r,x)|^{2n}\\
                        &=\int\limits_{\mathbb{R}^{2nd}}\prod\limits_{j=1}^{n}[(e^{i\xi_j\cdot h}-1)(e^{-i\zeta_l\cdot h}-1)]\mathbb{E}[\prod\limits_{j=1}^n\hat{\phi}^\theta(z,r,\xi_j)\prod\limits_{l=1}^n\hat{\phi}^{\theta *}(z,r,\zeta_l)]e^{i\sum_{j=1}^n(\xi_j-\zeta_j)\cdot x}\prod\limits_{j=1}^n\frac{\mathrm{d}\xi_j\mathrm{d}\zeta_j}{(2\pi)^{2d}}\,.
        \end{aligned}
        \end{equation*}
        Note that
        \begin{equation*}
            \hat{\phi}^\theta(z,r,\xi) =\int\limits_{\mathbb{R}^d}u^\theta(z,\eps^{-\beta}r+\eta x)e^{-i\xi\cdot x}\mathrm{d}x=\frac{1}{\eta^d}\hat{u}^\theta(z,\eta^{-1}\xi)e^{\frac{i}{\eps^\beta\eta}r\cdot\xi}\,.
        \end{equation*}
    From Theorem \ref{thm:psi_theta_decomp}, we deduce that
    \begin{equation}\label{eqn:phi_hat_expn}
    \begin{aligned}
        \mathbb{E}[\prod\limits_{j=1}^n\hat{\phi}^\theta(z,r,\xi_j)\prod\limits_{l=1}^n\hat{\phi}^{\theta *}(z,r,\zeta_l)]  &=\frac{1}{\eta^{2nd}}e^{\frac{i}{\eps^\beta\eta}r\cdot\sum_{j=1}^n(\xi_j-\zeta_j)}\hat{\mu}^\theta_{n,n}(z,\eta^{-1}v)\\
        &=\frac{1}{\eta^{2nd}}e^{\frac{i}{\eps^\beta\eta}r\cdot\sum_{j=1}^n(\xi_j-\zeta_j)}\Pi^\theta_{n,n}(z,\eta^{-1}v)[\mathcal{B}^\theta_{n,n}(z,\eta^{-1}v)+\mathscr{R}^\theta_{n,n}(z,\eta^{-1}v)]\,.
    \end{aligned}
    \end{equation}
    
Note that 
\begin{equation*}
    \Pi^\theta_{n,n}(z,v)\mathcal{B}^\theta_{n,n}(z,v)=\mathbb{E}[\prod\limits_{j=1}^n\hat{u}^\eps(z,\xi_j)\hat{u}^{\eps *}(z,\zeta_j)]\,,
\end{equation*}
where $\hat{u}^\eps$ is the Fourier transform of $u^\eps$, the solution to the It\^o-Schr\"{o}dinger equation \eqref{eq:uSDEd} recast when $2k_0=1, R\to 16R$ as
\begin{eqnarray*}
    \mathrm{d}u^\eps=\frac{i\eta}{\eps}\Delta_xu^\eps\mathrm{d}z-\frac{R(0)}{2\eta^2}u^\eps\mathrm{d}z+\frac{i}{\eta}u^\eps\mathrm{d}B,\quad u^\eps(0,x)=u_0^\eps(x)\,.
\end{eqnarray*}
So the first term in the sum~\eqref{eqn:phi_hat_expn} corresponds to the moments of the It\^o-Schr\"{o}dinger equation. Performing an inverse Fourier transform on this gives, using the shorthand notation $\mathrm{d}v=\prod_{j=1}^n (2\pi)^{-2d} \mathrm{d}\xi_j \mathrm{d}\zeta_j$:
\begin{equation*}
    \begin{aligned}
        &\frac{1}{\eta^{2nd}}\int\limits_{\mathbb{R}^{2nd}}\prod\limits_{j=1}^{n}[(e^{i\xi_j\cdot h}-1)(e^{-i\zeta_l\cdot h}-1)]e^{i\sum_{j=1}^n(\xi_j-\zeta_j)\cdot x}e^{\frac{i}{\eps^\beta\eta}r\cdot\sum_{j=1}^n(\xi_j-\zeta_j)}\mathbb{E}[\prod\limits_{j=1}^n\hat{u}^\eps(z,\eta^{-1}\xi_j)\hat{u}^{\eps *}(z,\eta^{-1}\zeta_j)]
        \mathrm{d}v
        \\
        &=\int\limits_{\mathbb{R}^{2nd}}\prod\limits_{j=1}^{n}[(e^{i\xi_j\cdot (\frac{r}{\eps^\beta}
        +\eta x+\eta h)}-e^{i\xi_j\cdot (\frac{r}{\eps^\beta}+\eta x)})(e^{-i\zeta_l\cdot(\frac{r}{\eps^\beta}+\eta x +\eta h) }-e^{-i\zeta_l\cdot(\frac{r}{\eps^\beta}+\eta x) })]\mathbb{E}[\prod\limits_{j=1}^n\hat{u}^\eps(z,\xi_j)\hat{u}^{\eps *}(z,\zeta_j)]
        \mathrm{d}v
        \\
        &=\mathbb{E}|\phi^\eps(z,r,x+h)-\phi^\eps(z,r,x)|^{2n}:=\mathbb{E}|\delta_h\phi^\eps(z,r,x)|^{2n}\,,
    \end{aligned}
\end{equation*}
    where $\phi^\eps(z,r,x)=u^\eps(z,\eps^{-\beta}r+\eta x)$ as in \eqref{eq:phitheta}.~\eqref{eqn:phi_hat_expn} can now be written as
        \begin{equation*}
            \begin{aligned}
              &  \mathbb{E}|\delta_h\phi^\theta(z,r,x)|^{2n}=\mathbb{E}|\delta_h\phi^\eps(z,r,x)|^{2n} + \\
              &\frac{1}{\eta^{2nd}}\int\limits_{\mathbb{R}^{2nd}}\prod\limits_{j=1}^{n}[(e^{i\xi_j\cdot h}-1)(e^{-i\zeta_l\cdot h}-1)]e^{\frac{i}{\eps^\beta\eta}r\cdot\sum_{j=1}^n(\xi_j-\zeta_j)}\Pi^\theta_{n,n}(z,\eta^{-1}v)\mathscr{R}^\theta_{n,n}(z,\eta^{-1}v)e^{i\sum_{j=1}^n(\xi_j-\zeta_j)\cdot x}
              \mathrm{d}v
              .
            \end{aligned}
        \end{equation*}
        The first term corresponds to moments from the It\^o-Schr\"{o}dinger equation which was analyzed in \cite[Theorem 2.7]{bal2024complex}.  The result there shows that 
        \begin{equation*}
           \sup_{s\in[0,z]} \mathbb{E}|\delta_h\phi^\eps(z,r,x)|^{2n}\le C(z,n)|h|^{2n\alpha}\,,
        \end{equation*}
        where $0<\alpha\le 1$ in the kinetic regime and $0<\alpha<1$ in the diffusive regime. 
    
    The second part can be bounded as follows. Performing a change of variables $\eta^{-1}v\to v$ shows that this is bounded by
        \begin{equation*}
            \begin{aligned}
    \int\limits_{\mathbb{R}^{2n}}\prod\limits_{j=1}^{n}[|e^{i\eta\xi_j\cdot h}-1||e^{-i\eta\zeta_l\cdot h}-1|]|\mathscr{R}^\theta_{p,q}(z,v)|\prod\limits_{j=1}^n\mathrm{d}\xi_j\mathrm{d}\zeta_j&\le \eta^{2n}|h|^{2n}\|\prod\limits_{m=1}^{2n}\langle v_m\rangle \mathscr{R}^\theta_{n,n}(z,v)\|\,.
            \end{aligned}
    \end{equation*}
        Due to Theorem~\ref{thm:psi_theta_decomp}, we have that this is bounded by
        \begin{equation*}
           |h|^{2n}\theta\eta^{2n-2}e^{\frac{\lambda[{\color{black}C},n,n,z,1]}{\eta^2}}\|\prod\limits_{m=1}^{2n}\langle v_m\rangle^2\psi^\theta(0,v)\|,
        \end{equation*}
        where from~\eqref{eqn:lambda_def}, ${\color{black}\lambda[C,n,n,z,1]=c(n,z)\int_{\mathbb{R}^{d+1}}\langle s\rangle\langle k\rangle^{2+2n}|\hat{C}(s,k)|\mathrm{d}s\mathrm{d}k}$. Due to assumption~\eqref{eqn:C_n_assump}, we can find an $n$ such that $2n\ge d+1$ and $\|\langle k\rangle^{2+2n}\hat{R}(k)\|$ is bounded so that the term above is bounded by $|h|^{2n}$ (and thereby $|h|^{2n\alpha}$) uniformly in $\theta$. 
        
       The tightness and stochastic continuity of the process $x\to\phi^\theta(z,r,x)$ in $C^{0,\alpha_{-}}(\mathbb{R}^d)$ is then a standard result; see, e.g., \cite{kunita1997stochastic}.
\end{proof}
{\subsection{Proof of Theorem~\ref{thm:phi_r_indep}}\color{black}
\begin{proof}
It suffices to verify that 
\begin{equation*}
    m^\theta_{p_1+p_2,q_1+q_2}(z,r,X,Y,r',X',Y')=\mathbb{E}\prod_{j=1}^{p_1}\phi^\theta(z,r,x_j)\prod_{j=1}^{p_2}\phi^\theta(z,r',x'_j)\prod_{l=1}^{q_1}\phi^{\theta\ast}(z,r,y_l)\prod_{l=1}^{q_2}\phi^{\theta\ast}(z,r',y'_l)
\end{equation*}
is asymptotically equal to 
\begin{equation*}
    m_{p_1,q_1}(z,r,X,Y)m_{p_2,q_2}(z,r',X',Y')=\mathbb{E}[\prod_{j=1}^{p_1}\phi(z,r,x_j)\prod_{l=1}^{q_1}\phi^\ast(z,r,y_l)]\mathbb{E}[\prod_{j=1}^{p_2}\phi(z,r',x'_j)\prod_{l=1}^{q_2}\phi^\ast(z,r',y'_l)]\,.
\end{equation*}
From Theorem~\ref{thm:mu_theta_decomp_final}, $m^\theta_{p_1+p_2,q_1+q_2}$ is well approximated by sums of products of second moments. We have that this is equal to the product $m_{p_1,q_1}(z,r,X,Y)m_{p_2,q_2}(z,r',X',Y')$ if the two-point correlation between points centred at $r$ and $r'$ are asymptotically zero. For simplicity, we only prove the case for when $\beta=1$. From Theorem~\ref{thm:mu_theta_decomp_final},
\begin{equation*}
    \begin{aligned}
        \mathbb{E}\phi^\theta(z,r,x)\phi^{\theta\ast}(z,r',y')=m_{1,1}^\theta(z,r,r',x,y')+c(z)(\eps^{\frac13}+\theta^{\frac12})\,,
    \end{aligned}
\end{equation*}
where the second moment $m_{1,1}^\theta$ is given explicitly by~\cite{bal2024complex}
\begin{equation*}
    \begin{aligned}
        m_{1,1}^\theta(z,r,r',x,y')&=\int\limits_{\mathbb{R}^{2d}}e^{i\eps^{-1}(\xi\cdot r-\zeta\cdot r')}e^{i\eta(\xi\cdot x-\zeta\cdot y')}e^{-\frac{i\eta z}{\eps}(|\xi|^2-|\zeta|^2)}\hat{u}^\eps_0(\xi)\hat{u}^{\eps\ast}_0(\zeta)\\
        &\times\exp\Big(\frac{z}{\eta^2}\int_0^1Q\big(\frac{r'-r}{\eps}+\eta(y'-x)+\frac{2\eta sz}{\eps}(\xi-\zeta)\big)\mathrm{d}s\Big)\frac{\mathrm{d}\xi\mathrm{d}\zeta}{(2\pi)^{2d}}\,,
    \end{aligned}
\end{equation*}
where $Q(x)=R(x)-R(0)\le 0$. From a change of variables $(\xi/\eps,\zeta/\eps)\to(\xi,\zeta)$, 
\begin{equation*}
    \begin{aligned}
        m_{1,1}^\theta(z,r,r',x,y')&=\int\limits_{\mathbb{R}^{2d}}e^{i(\xi\cdot r-\zeta\cdot r')}e^{i\eta\eps(\xi\cdot x-\zeta\cdot y')}e^{-{i\eta\eps z}(|\xi|^2-|\zeta|^2)}\hat{u}_0(\xi)\hat{u}^{\ast}_0(\zeta)\\
        &\times\exp\Big(\frac{z}{\eta^2}\int_0^1Q\big(\frac{r'-r}{\eps}+\eta(y'-x)+{2\eta sz}(\xi-\zeta)\big)\mathrm{d}s\Big)\frac{\mathrm{d}\xi\mathrm{d}\zeta}{(2\pi)^{2d}}\,.
    \end{aligned}
\end{equation*}
Noting that 
\begin{equation*}
    \begin{aligned}
        \lim_{\eps\to 0}\exp\Big(\frac{z}{\eta^2}\int_0^1Q\big(\frac{r'-r}{\eps}+\eta(y'-x)+{2\eta sz}(\xi-\zeta)\big)\mathrm{d}s\Big)\to 0
    \end{aligned}
\end{equation*}
pointwise, the proof follows from Lebesgue's dominated convergence theorem.
\end{proof}
}
\section{Conclusions}
Starting from the paraxial approximation, we have characterized the limiting statistical distribution of high frequency wavefields propagating over long distances through weak turbulence. In the scintillation regime considered here, we show that an appropriately rescaled macroscopic wavefield is asymptotically described by a complex Gaussian distribution with scintillation saturating to unity. This is consistent with experimental observations of speckle patterns where intensity of the wavefield is modeled by an exponential law. We have used a Gaussian assumption with short range correlations on the refractive index fluctuations to simplify our derivations. This allows us to show that the statistical moments of the wavefield of the paraxial wave equation are well approximated by moments given by closed-form PDEs of the It\^{o}-Schr\"{o}dinger model. Our analysis is primarily done in the Fourier domain with errors controlled in the total variation norm which translates to error estimates in the uniform sense in the physical domain. This choice is important for applications as physical manifestations like speckle are observed predominantly in the spatial domain. This in turn allows us to use the results in~\cite{bal2024complex} where the moments of the It\^{o}-Schr\"{o}dinger equation have been shown to converge to that of a complex Gaussian random variable. The extension of these results to a more general class of turbulence scenarios where the Gaussian assumption on the randomness is relaxed {\color{black}and numerical simulations in the scintillation regime, say using operator splitting techniques~\cite{bal2025splitting} are} left for future investigation.

\section*{Acknowledgments} We would like to thank Christophe Gomez for many stimulating discussions on the validity of the paraxial and It\^o-Schr\"odinger models. {\color{black}We also thank the two anonymous referees for their feedback in helping relax the assumptions on the random potential, suggesting the inclusion of Theorem~\ref{thm:phi_r_indep} and in improving the overall presentation of this paper.}
\bibliographystyle{siam}
\bibliography{Reference}

\end{document}